\documentclass[10pt,a4paper]{article}

\usepackage[T1]{fontenc}
\usepackage[all]{xy}
\usepackage{amssymb}
\usepackage{amsmath}
\usepackage{bbm}
\usepackage{graphics,graphicx}
\usepackage{ esint }
\usepackage{eepic}
\usepackage{epsfig}
\usepackage{times}
\usepackage{stmaryrd}
\usepackage{hyperref}

\usepackage[usenames]{color}
\usepackage[toc,page]{appendix}
\DeclareMathOperator{\arcsinh}{arcsinh}

\newtheorem{theorem}{Theorem}[section]
\newtheorem{lemma}[theorem]{Lemma}
\newtheorem{prop}[theorem]{Proposition}

\newtheorem{remark}[theorem]{Remark}
\newtheorem{corollary}[theorem]{Corollary}

\newenvironment{proof}{\noindent\text{\textbf{Proof.\:}}}{}

\def\qed{\hfill $\square$ \goodbreak \smallskip}

\def\eps{\varepsilon}
\def\R{{\mathbb R}}
\def\N{{\mathbb N}}
\def\S{{\mathbb S}}

\def\Om{\Omega}
\def\C{{\mathcal C}}
\def\D{{\mathcal D}}
\def\F{\mathcal{F}}
\def\H{\mathcal{H}}

\def\div{{\rm div}}
\def\vol{{\rm Vol}}

\def\Ind{\raise 2pt\hbox{$\chi $}}

\newcounter{mnotecount}[section]

\newcommand{\rmnote}[1]{}

\newbox\boxint\newbox\boxtir\newdimen\longa\newdimen\longb
\def\intgm{
\setbox\boxint=\hbox{$\displaystyle\int$}
\setbox\boxtir=\hbox{$-$}
\longa=\wd\boxint\longb=\wd\boxtir
\advance\longa-\longb\div_{g}ide\longa2
\advance\longb\longa
\box\boxint\hskip-\longb\box\boxtir}

\def\intm{
\setbox\boxint=\hbox{$\int$}
\setbox\boxtir=\hbox{{--}}
\longa=\wd\boxint\longb=\wd\boxtir
\advance\longa-\longb\div_{g}ide\longa2
\advance\longb\longa
\box\boxint\hskip-\longb\box\boxtir}
\def\moyp{\mathop{\intm}\nolimits}
\def\moy_#1{\ifmmode\ifinner{\moyp_{#1}}\else{\fint_{#1}}\fi\fi}

\title{Existence and regularity of Faber-Krahn minimizers\\ in a Riemannian manifold}
\author{Jimmy Lamboley\footnote{Sorbonne Universit\'e, Universit\'e Paris Diderot, CNRS, Institut de Math\'ematiques de Jussieu-Paris Rive Gauche, IMJ-PRG, F-75005, Paris, France. E-mail: \texttt{jimmy.lamboley@imj-prg.fr}}, Pieralberto Sicbaldi\footnote{Universidad de Granada, Departamento de Geometr{\'{\i}}a y Topolog{\'{\i}}a, Facultad de Ciencias, Campus Fuentenueva, 18071 Granada, Spain \& Aix Marseille Universit\'e, CNRS, Centrale Marseille, I2M, Marseille, France.
E-mail: \texttt{pieralberto@ugr.es}}}
\begin{document}

\maketitle
\begin{abstract}

In this paper, we study the minimization of $\lambda_{1}(\Om)$, the first Dirichlet eigenvalue of the Laplace-Beltrami operator, within the class of open sets $\Omega$ of fixed volume in a Riemmanian manifold $(M,g)$. In the Euclidian setting (when $(M,g)=(\R^n,e)$), the well-known Faber-Krahn inequality asserts that the solution of such problem is any ball of suitable volume. Even if similar results are known or may be expected for Riemannian manifolds with symmetries, we cannot expect to find explicit solutions for general manifolds $(M,g)$. In this paper we study existence and regularity properties for this spectral shape optimization problem in a Riemannian setting, in a similar fashion as for the isoperimetric problem. We first give an existence result in the context of compact Riemannian manifolds, and we discuss the case of non-compact manifolds by giving a counter-example to existence. We then focus on the regularity theory for this problem, and using the tools coming from the theory of free boundary problems, we show that solutions are smooth up to a possible residual set of co-dimension 5 or higher.

\medskip

{\it Keywords:\,} Shape optimization, Laplace-Beltrami operator, first eigenvalue, regularity of free boundaries, Riemannian manifold, Faber-Krahn profile, isoperimetric problems.
\smallskip

\end{abstract}


\section{Introduction and main results}

Let $(M,g)$ be a  smooth $n$-dimensional Riemannian manifold (without boundary), where $n \geq 2$.
For all open subset $\Omega$ of $M$, we denote by $\lambda_1(\Omega)$ the first eigenvalue of the Laplace-Beltrami
operator $\Delta_g$ in $\Omega$, with zero Dirichlet boundary conditions on $\partial\Om$, that is
\begin{equation}\label{eq:lambda1}
\lambda_{1}(\Om)=\displaystyle \min\left\{\frac{\displaystyle \int_{\Om}\|\nabla^{g} u\|_{g}^2\, \textrm{dvol}_g}{\displaystyle \int_{\Om}u^2\,\textrm{dvol}_g}, \;\;u\in H^1_{0}(\Om)\right\},
\end{equation}
where $\textrm{dvol}_g$, $\| \cdot \|_g$ and $\nabla^g$ represent respectively the volume form, the norm and the gradient, all with respect to the metric $g$. The Sobolev space $H_0^1(\Omega)$ also refers to the metric $g$.
When $\Om$ is smooth enough, we can characterize $\lambda_{1}(\Om)$ by the existence of $u_{\Om}$  such that
\begin{equation}\label{eq:edp}
\left\{
\begin{array}{ll}
\Delta_g u_{\Omega}+\lambda_1(\Omega)\, u_{\Omega}=0&\textrm{in}\;\;\Omega,\\[2mm]
u_{\Omega}=0&\textrm{on}\;\;\partial\Omega,
\end{array}
\right.\quad \textrm{ with } \quad u_{\Om}\geq 0 \quad \textrm{ and }\quad \displaystyle{\int_\Omega u_\Omega^2 \, \textrm{dvol}_g}=1,
\end{equation}
and $u_{\Om}$ is called the first normalized eigenfunction of the Laplace-Beltrami operator with Dirichlet boundary condition on $\partial\Om$.

\medskip

Let $\vol_g(M)$ denote the volume of the Riemannian manifold $M$, that can be infinite. We are interested in  the existence and the regularity of optimal sets for the following shape optimization problem: for any $m\in(0,\vol_g(M))$, find an open subset $\Omega^*\subset M$ of volume $m$ such that
\begin{equation}\label{eq:pbforme}
\displaystyle{\lambda_1(\Omega^*)=\min\{\lambda_1(\Omega);\;\Omega\, \textrm{open subset of}\, M, \vol_g(\Omega)=m\}.}
\end{equation}
The solutions $\Omega^*$ of such optimization problem are called {\it Faber-Krahn minimizers}, and the function 
\[
FK:m\in (0,\vol_{g}(M))\mapsto FK(m)
\] 
associating to $m$ the value of the infimum in \eqref{eq:pbforme}, is called the {\it Faber-Krahn profile} of the manifold $(M,g)$.

\medskip

This problem is inspired by the classical {\it isoperimetric problem}: for any $m\in(0,\vol_g(M))$, find an open domain $\Omega\subset M$ whose boundary $\Sigma = \partial \Omega$ minimizes area among regions of volume $m$. The region $\overline{\Omega}$ and its boundary $\Sigma$ are
called {\it isoperimetric region} and {\it isoperimetric hypersurface} respectively. In the Euclidean space, $\Omega$ must be a ball by the standard isoperimetric inequality, and $\Sigma$ is then a sphere. For a general Riemannian manifold this fact fails, and the shape of the optimal region can be very difficult to understand. Nevertheless, the following fundamental results about the existence and the regularity of isoperimetric regions are now very well-known: by the seminal papers of Almgren \cite{Alm}, Gr\"uter \cite{Gru}, and Gonzalez, Massari,
Tamanini \cite{GMT}, if $M$ is a compact $n$-dimensional Riemannian manifold, then, for any positive $m<\vol_g(M)$, there exists an open set $\Omega \subset M$ whose boundary $\Sigma$ minimizes area among regions of volume $m$, and, except for a closed singular set of Hausdorff dimension at most $n-8$, $\Sigma$ is a smooth embedded hypersurface with constant mean curvature. In particular, for dimensions of the ambient manifold less or equal to 7, the isoperimetric hypersurface $\Sigma$ (that is an objet of dimension $n-1$, then less or equal to 6) is an embedded hypersurface. In fact, an isoperimetric hypersurface $\Sigma$ has an area-minimizing tangent cone at each point, and if a tangent
cone at $p\in \Sigma$ is an hyperplane, then $p$ is a regular point of $\Sigma$. The value of the critical dimension (i.e. 8 if we consider the dimension of the ambient manifold, and 7 if we consider the dimension of the isoperimetric hypersurface) relies on the existence of the Simons cone in $\R^8$, i.e. 
\[
C = \{ (x_1, ..., x_8) \in \R^8 \, \, \, |\, \, \, x_1^2+ x_2^2 +x_3^2+x_4^2= x_5^2+x_6^2+x_7^2+x_8^2\}
\]
which is a global minimizer of the area functional.

\medskip

Coming back to the problem of finding Faber-Krahn minimizers, when the manifold is the Euclidean space it is well known that a ball $\Om^*$ of volume $m>0$ is a solution to the problem (\ref{eq:pbforme}) (as for the isoperimetric problem), and in particular it exists and it is smooth. This fact follows from the Faber-Krahn inequality
: for all open subset $\Omega$ of $\mathbb{R}^n$ whose volume is $m$, we have
\begin{equation} \label{e:faber-krahn}
\lambda_1(\Omega) \,\ge\, \lambda_1(B_m) \,
\end{equation}
where $B_m$ is a round ball in~$\mathbb{R}^n$ with volume $m$; moreover equality holds in~\eqref{e:faber-krahn} if and only if 
$\Omega = B^m$ up to translation and to sets of 0 capacity. 
 We point out that the proof of the Faber-Krahn inequality relies on the isoperimetric inequality.

When $(M,g)$ is a general Riemannian manifold with no symmetry, we cannot expect to explicitely identify minimizers for problem \eqref{eq:pbforme}, and very few results are known.  Nevertheless, the Faber-Krahn profile and the isoperimetric profile are linked (see for example \cite{C84Eig}), and starting from the analogy with the isoperimetric problem, in the following papers there is a construction of examples of domains that are critical for $\Omega \mapsto \lambda_1(\Omega)$ under volume constraint in some Riemannian manifolds, but it is not known a priori if such critical domains are or not Faber-Krahn minimizers: \cite{P1, P2, P3, P4, P5, P6}. Notice that when the manifold has no symmetry, the constructions of such examples is limited to small or big volumes. Moreover, all such examples have regular boundary.
\medskip

In this paper, we are inspired by the existence and regularity results for the isoperimetric problem recalled before, and we plan to obtain similar results for the Faber-Krahn problem \eqref{eq:pbforme}. Nevertheless, for the existence, it is now classical (see for example \cite{B07Doo}) that one cannot expect to prove directly that there exists an open set solution of \eqref{eq:pbforme}. Indeed, such a result (first part of Theorem \ref{th:reg0}, which will be proven in Section \ref{ssect:lip}) is already a regularity result for an optimal set; the reason is that the class of open sets does not satisfy any suitable compactness property for our problem. The usual way to overcome this difficulty is to relax our minimization problem in the class of \textit{quasi-open sets}, relying on the notion of capacity. Basically, quasi-open sets are level sets of functions of $H^1(M)$, which happen to not be necessarily continuous, so their level sets aren't necessarily open (however, any open set is quasi-open); for more details on the study of capacity and quasi-open sets, see for example  \cite{EG92Mea,HP05Var}. If $\Om\subset M$ is a quasi-open set, we can define
\begin{equation}\label{eq:H10}
H^1_{0}(\Om)=\{u\in H^1_{0}(M),\;u=0\; q.e. \textrm{ on }M\setminus\Om\}.
\end{equation}
where $q.e.$ means quasi-everywhere, which means everywhere except on a set of capacity 0  (see Section 2 for the definition of the capacity in the Riemannian setting). This definition retrieves the usual definition of $H^1_{0}(\Om)$ when $\Om$ is an open set, namely the closure of $C^\infty_{c}(\Om)$ for the $H^1$-norm. Once we have a definition of the space $H^1_{0}(\Om)$, definition \eqref{eq:lambda1} can be applied to define $\lambda_{1}(\Om)$ for any quasi-open set $\Om$, and it is classical that equation \eqref{eq:edp} has a meaning in the weak sense, in particular $u_{\Om}\in H^1_{0}(\Om)$ solution to \eqref{eq:edp} exists and is unique.

 For the regularity of Faber-Krahn minimizers, as it happens for the isoperimetric hypersurfaces, we will prove that there exists a critical dimension $k^*$ for which Faber-Krahn minimizers are regular if $n<k^*$, and singularities can appear starting from the dimension $n=k^*$. In order to define this critical dimension $k^*$, we need to recall 
notion of homogeneous global minimizer of the Alt-Caffarelli functional in $\R^n$, that is a homogeneous function $u_0\in H^1_{loc}(\R^n)$ such that:
\begin{equation}
\int_{B_{R}(0)}|\nabla^e u_0|^2 + \vol_e(\{u_0>0\}\cap B_{R}(0)) \le\int_{B_{R}(0)}|\nabla^e w|^2 + \vol_e(\{w>0\}\cap B_{R}(0)).
\end{equation}
for every $R>0$ and $w\in H^1_{loc}(\R^n)$ such that $w=u_0$ outside $B_{R}(0)$ (the Euclidean ball of radius $R$ and center $0$), where $\nabla^e$ is the Euclidean gradient and $\vol_{e}$ is the Lebesgue measure in $\R^n$. We will see that the Alf-Caffarelli functional plays in our problem the same role as the area functional in the isoperimetric problem, and $k^*$ can be defined as the smallest integer such that there exists a non-trivial homogeneous global minimizer of the Alt-Caffarelli function. Unfortunately, finding the exact value of $k^*$ is still a difficult open problem. Nevertheless, thanks to the important results by Caffarelli, Jerison, Kenig \cite{CJK04Glo}, De Silva, Jerison \cite{DJ08Asi}, and Jerison, Savin  \cite{JS15Som} we know such dimension belongs to the set $\{5,6,7\}$.

\medskip

We are now in position to state our main results. The first one is the following:

\begin{theorem}\label{th:existence0} (Existence)
If $M$ is compact
 and $m\in(0,\vol_g(M))$, then there exists a quasi-open set $\Omega^*$ solution of 
\begin{equation}\label{eq:pbforme2}
\displaystyle{\lambda_1(\Omega^*)=\min\{\lambda_1(\Omega);\;\Omega\, \textrm{quasi-open subset of}\, M, \vol_g(\Omega)=m\}.}
\end{equation}
\end{theorem}

\medskip

Notice that in this result we need the compactness of the manifold $M$. In Section \ref{ssect:nonexist}, we discuss the compactness hypothesis, and we exhibit a non-compact manifold $M$ so that for any parameter $m>0$, problems \eqref{eq:pbforme} and \eqref{eq:pbforme2} do not have solutions. Compactness in not required to state the regularity result about Faber-Krahn minimizers. Nevertheless we need an other important topological assumption that is the connectedness (the discussion of the connectedness hypothesis of the manifold $M$ is done in Remark \ref{rk:connectedness} that follows). Our second main result is then the following:

\begin{theorem}\label{th:reg0} (Regularity)
Let $\Om^*$ be a solution of (\ref{eq:pbforme2}) for $m\in(0,\vol_g(M))$, and assume $M$ is connected. Let $k^*$ be the lowest dimension $k$ such that there exists a non-trivial homogeneous global minimizer of the Alt-Caffarelli functional in $\R^k$ (it is known that $k^*\in\llbracket5,7\rrbracket$). Then:
\begin{enumerate}
\item  $\Omega^*$ is open (and therefore solves \eqref{eq:pbforme}) and has finite perimeter in $M$.
\item We can decompose $\partial\Omega^*=\Sigma_{reg}\cup\Sigma_{sing}$ in two disjoint sets such that:
\begin{enumerate} 
\item\label{item:reg} $\Sigma_{reg}$ is relatively open in $\partial\Om$ and is a smooth hypersurface in $M$ ($C^\infty$ if $M$ is $C^\infty$, analytic if $M$ is analytic),
\item\label{item:sing} we have:
\begin{itemize}
\item if $n< k^*$, then $\Sigma_{sing}=\emptyset$,
\item if $n=k^*$, then $\Sigma_{sing}$ is made of isolated points, 
\item if $n>k^*$, then $dim_{\H}(\Sigma_{sing})\leq n-k^*$, i.e.
\begin{equation}\label{eq:k*}
\forall s>n-k^*, \H^{s}(\Sigma_{sing})=0.
\end{equation}
\end{itemize}
\end{enumerate}
\end{enumerate}
\end{theorem}

Combining these two results, we get that for any $n$-dimensional connected and compact Riemannian manifold $M$, one can find an open set $\Om^*$ solution of \eqref{eq:pbforme}, which is $C^\infty$, up to a singular set of dimension less than $n-5$.

\begin{remark}\label{rk:connectedness}{\rm
Without connectedness assumption for the manifold $M$, regularity of a minimizer may fail. Consider $M$ to be the union of two disjoint copies of unit spheres $\S_{1}^{n}\cup\S_{2}^n$ endowed with its usual metric $g$; for 
\[
m\in (\vol_{g}(\S^n),2\vol_{g}(\S^n)=\vol_{g}(M))\,,
\] 
any set of the form $\Om=\S_{1}^n\cup\omega$ where $\omega$ is any quasi-open subset of $\S^n_{2}$ of volume $m-\vol_{g}(\S^n)$ is a solution to \eqref{eq:pbforme2} because $\lambda_1(\S_{1}^n\cup\omega) = \lambda_1(\S_{1}^n) < \lambda_1(\omega)$, and to \eqref{eq:pbforme} if $\omega$ is open, so one cannot expect any regularity property. Nevertheless, by replacing $\omega$ by any other smooth set of same volume, it is not hard to see that there still exists a smooth solution to \eqref{eq:pbforme}, see also \cite[Appendix]{BL09Reg}.}
\end{remark}

%
%

Let us discuss the strategy for proving these results, and their relation to the state of the art.

\medskip

\noindent{\bf About Theorem \ref{th:existence0}.} In the Euclidian setting, while the ball is known to be a solution, the problem retrieves its interest if one consider an extra ``box constraint'' of the form $\Om\subset D$ where $D$ is an open and bounded subset of $\R^n$. In this setting, existence results were obtain for problem \eqref{eq:pbforme} with two different strategies in \cite{H00Sur} and \cite{BHP05Lip}. The main difficulty for these results is to obtain a solution that is an open set. We focus first only on proving that there exists a quasi-open set solution to \eqref{eq:pbforme2}; the fact that solutions are open will be dealt with in Theorem \ref{th:reg0}.

Similarly to \cite{BHP05Lip}, we use the variational formulation of $\lambda_{1}(\Om)$ to show that problem \eqref{eq:pbforme2} is equivalent to solving a free boundary problem (namely \eqref{eq:pbfonc}), which is a calculus of variation problem consisting in the minimization of an energy $J(u)$ involving the level set $\Om_{u}=\{u\neq 0\}$, among functions $u\in H^1(M)$. Considering a solution $u$ to this free boundary problem, the set $\Om_{u}$ will be a solution to \eqref{eq:pbforme2}. Once we obtain a free boundary formulation, one can use the same strategy as in the seminal paper \cite{AC81Exi} of Alt and Caffarelli (see below for more details) to prove existence of a solution, which relies on classical tools of calculus of variation.

This strategy may fail if the manifold $M$ is not compact, and, as we said before, we give in Section 3 an explicit example of manifold $M$ so that \eqref{eq:pbforme} has no solution. In order to exhibit such a manifold, we will need two essential properties:
\begin{itemize}
\item the Faber-Krahn profile of the manifold $(M,g)$ is strictly bounded from below by the Faber-Krahn profile of the euclidian space $(\R^n,e)$; this happens to be true if the same is valid for the isoperimetric profile, see Proposition \ref{p1};
\item $M$ is asymptotically Euclidian in the sense that a geodesic ball in $M$ converging to infinity is a smooth perturbation of a euclidian ball of same volume.
\end{itemize}
We show that both properties are valid when $M$ is the usual catenoid in $\R^3$, which provides the expected counter-example to existence of Faber-Krahn minimizers (Theorem \ref{th:nonexist}).\\

\noindent{\bf About Theorem \ref{th:reg0}.} 
Regularity results for such kind of shape optimization problems are quite involved and will rely on many steps that we will describe in the introduction of Section \ref{sect:reg}. Similarly to the existence result, we start with formulation \eqref{eq:pbfonc0}  and the definition of the functional $J$ in \eqref{eq:pbfonc} introduced in Section \ref{sect:existence}, and we are naturally led to the field of ``regularity of free boundaries''. In order to understand our strategy, let us start by commenting on the extensive litterature on this topic: the regularity theory for such problems was initiated in \cite{AC81Exi}, where the authors study the regularity of the free boundary for 
\begin{equation}\label{eq:AC0}
\min \left\{\int_{D}|\nabla^e v|^2 +  \gamma\vol_{e}(\{v>0\})\;;\; v\in H^1(D), v=u_{0}\in \partial D\right\},
\end{equation}
where $u_{0}\geq 0$ is given and $D$ is an open bounded set in $\R^n$. 
In particular, the authors show that an optimal solution $v$ is locally Lipschitz continuous inside $D$, which is the optimal regularity one can expect for $v$, and implies in particular that $\{v>0\}$ is an open set. They show then that the free boundary $\partial\{v>0\}\cap D$ can be decomposed into a smooth (analytic) part $\Sigma_{reg}$ where one can write the classical optimality condition $|\nabla^e v|_{|\Sigma_{reg}}^2=\gamma$, and a (possibly) singular part $\Sigma_{sing}$ which is small in the sense that $\H^{n-1}(\Sigma_{sing})=0$ (where $\H^s$ denotes the $s$-dimensional Hausdorff measure); they also show  that in fact $\Sigma_{sing}=\emptyset$ if $n=2$.
These results have been improved by Weiss in \cite{W99Par} who introduced a monotonicity formula to study blow-up limits, which combined to the study of global homogeneous minimizer of the Alt-Caffarelli functional lead to the estimate $\dim_{\H}(\Sigma_{sing})\leq n-5$ if $n\geq 5$, and $\Sigma_{sing}=\emptyset$ if $n<5$.

We will apply a similar strategy for solutions to \eqref{eq:pbfonc0}-\eqref{eq:pbfonc} (which are the free boundary formulations of our intial problem \eqref{eq:pbforme2}), with three main differences that we need to take into account:
\begin{itemize}
\item deal with the term $\int_{M}w^2$ in $J(w)$, coming from the fact that we are dealing with an eigenvalue problem,
\item 
deal with the Riemannian metric $g$; if $g_{ij}(x)$ is the matrix of the coefficients of the metric $g$ in some suitable local coordinates system, from the PDE point of view to deal with $g$ replaces the Euclidian Laplace operator by an operator of the form 
\[
h\, \div(A\, \nabla\cdot)
\]
where $h(x) = \frac{1}{\sqrt{|g_{ij}(x)|}}$ and $A = g^{ij}(x) \, \sqrt{|g_{ij}(x)|}$, being $|g_{ij}(x)|$ and $g^{ij}(x)$ respectively the determinant of the matrix $g_{ij}(x)$ and the inverse matrix of the matrix $g_{ij}(x)$. 
\item handle the volume constraint instead of a penalization of the volume as in \eqref{eq:AC0}. 
\end{itemize}

Let us mention a few important contributions in similar developements. In \cite{BL09Reg},  T. Brian\c{c}on and the first author of the present paper managed to overcome the first and third difficulties in the Euclidian setting. Note however that they only adapted the result by Alt and Caffarelli and did not adapt the improvement given by Weiss, therefore the estimate of the singular set they obtain was not optimal. In other words, even in the Euclidian setting, Theorem \ref{th:reg0} improves the results in \cite{BL09Reg} (of course, one could argue that solutions are Euclidian balls in this context, but as in \cite{BL09Reg}, one can consider a box constraint of the type $\Om\subset D$ where $D\subset M$, so that it may happen that balls are not admissible sets; even if we did not take into account this constraint in the current paper, when we are concerned with the regularity inside the box (far from $\partial D$), since all argument are local, Theorem \ref{th:reg0} remains valid in this case).

In \cite{W05Opt}, A. Wagner did study the first steps of the strategy from \cite{AC81Exi} for a problem similar to \eqref{eq:pbfonc} (in the Euclidian setting but with an operator of the form $\div(A\nabla\cdot)$). He studies a penalized version of the problem, in a similar fashion to \cite{AAC86Ano}, which leads to the existence of a solution to \eqref{eq:pbforme} (in particular, it is an open set), enjoying some density estimates and a weak formulation of the optimality condition (named ``weak solutions'' in \cite{AC81Exi}). Again, our result is an improvment in the sense that we show that every solution is an open set, and we improve their regularity properties.

More recently, in \cite{DT15Reg,DET17Fre}, the authors develop a regularity theory  for quasi-minimizers of the Alt-Caffarelli functional, including the improvment given by the Weiss-monotonicity formula (therefore, leading to a similar regularity as in Theorem \ref{th:reg0}). However, it is not true that minimizers we are interested in are quasi-minimizers in the sense of \cite{DET17Fre}, as one cannot see the Laplace-Beltrami operator as a small deformation of the Euclidian Laplacian. A similar regularity theory for $\div(A\nabla\cdot)$ operators has been started in \cite{QT18Alm}, though they only deal with the first step of the strategy in proving that optimal solutions are H\"older-continuous in the general case.
But even if a similar regularity theory was valid for such elliptic operator, it would still remain the difficulty to prove that a solution to \eqref{eq:pbfonc} is a quasi-minimizer in the sense of \cite{DET17Fre}, the main difficulty here being to handle the volume constraint. This seems to be a significant and important open problem.

The last contributions we would like to mention are \cite{MTV17Reg,RTV18Exi} which were a strong inspiration for our work. The results in \cite{RTV18Exi} are similar to ours in the sense that they extend results for the Euclidean Laplace operator to a more general class, though the author deal with a drifted operator of the form $-\Delta+\nabla\Phi\cdot\nabla$. Nevertheless, as we have to deal with a Laplace-Beltrami operator, several steps and ideas differ from the current paper.\\

\noindent{\bf Description of the paper.}
In the following section, we introduce the Riemannian setting of our problem. In the third section we introduce the main free boundary formulation which is in some sense equivalent to our shape optimization problem \eqref{eq:pbforme2}, and we use this formulation to prove Theorem \ref{th:existence0}. We also exhibit a non-compact manifold leading to a non-existence phenomenon. Finally, in Section 4, we prove Theorem \ref{th:reg0}.

%
%
%
%
%
%

\section{Basic Riemannian notations}

 Since one of the goals of this paper is to bring the regularity theory for free boundary problems to Riemannian manifold, and then to mix together the geometric and the analytic language, we think it can be convenient for the reader to fix the basic Riemannian notation. For a more complete presentation see \cite{C84Eig}. 
Let $M$ be a Riemannian manifold with metric $g$. If $p \in M$ and $f$ is a $C^1$ real function defined in a neighborhood of $p$, we will denote by $\nabla^g f(p)$ the gradient of $f$ at $p$, i.e. the only vector of the tangent space $T_pM$ such that
\[
g(\nabla^g f(p), \nu_p) = \nu_p f
\]
for every vector $\nu_p \in T_pM$, where $\nu_p f$ is the directional derivative of $f$ at $p$ in the direction $\nu_p$; $\nabla^g f$ will be the gradient vector field, i.e. $\nabla^g f \in T M$, where $T M $ is the tangent bundle of $M$. If $\nabla_\nu X$ is the covariant derivative of a vector field $X$ on the manifold $M$ with respect to $\nu \in TM$, the divergence of $X$ is defined as
\[
\textnormal{div}_g X = \textnormal{trace} (\nu \to \nabla_\nu X),
\]
and the generalization of the Laplacian on a manifold, known as the Laplace-Beltrami operator, is defined by 
\[
\Delta_g f = \textnormal{div}_g(\nabla^g f),
\]
where $f$ is supposed to be of class $C^2$.
Let $U$ be an open set of $M$ and $\phi : U \to \mathbb{R}^n$ a chart on $M$, i.e. a diffeomorphism of $U$ into $\mathbb{R}^n$. If $(x^1,...,x^n)$ are the local coordinates and $\frac{\partial}{\partial x^i}$ are the coordinate vector fields, $i = 1,...,n$, we can define the matrix
\[
G = \left(g_{ij}\right)_{i,j = 1, ..., n}
\]
where $g_{ij} = g\left(\frac{\partial}{\partial x^i}, \frac{\partial}{\partial x^j}\right)$ are the coefficients of the metric $g$, which can be written as
\[
g = g_{ij}(x)\, \textnormal{d}x^i \,  \textnormal{d}x^j
\]
using the Einstein summation convention. 
We denote by $|g|$ the determinant of $G$ and by $g^{ij}$ the coefficients of $G^{-1}$, the inverse matrix of $G$. Straightforward computations show:
\[
\nabla^g f = g^{ij} \, \partial_j f
\]
and
\[
\Delta_g f = \frac{1}{\sqrt{|g|}}\, \partial_i \left( g^{ij}\, \sqrt{|g|}\, \partial_j f \right) = g^{ij} \, \partial_i \partial_j f + \partial_i\, g^{ij} \, \partial_j f + \frac{1}{2} \,  g^{ij} \, \partial_i (\log |g|) \, \partial_j f
\]
where $\partial_j$ denotes the standard derivation with respect to $x^j$. We will denote $\textnormal{dvol}_g$ the Riemannian measure associate to $g$, that in $U$ is
\[
\textnormal{dvol}_g = \sqrt{|g|} \, \textnormal{d}x
\]
where $\textnormal{d}x$ is the Lebesgue measure on $\phi(U)$. More generally, if $\{\phi_i : U_i \to \mathbb{R}^n\}$ is a chart covering of $M$, with an associate partition of unity $\{b_i\}$, we have
\[
\textnormal{dvol}_g = \sum_{i} b_i\, \sqrt{|g|} \, \textnormal{d}x\,
\]
in every chart.  A function $f: M \to \mathbb{R}$ is measurable if it is measurable in every chart, and the definition of $\textnormal{dvol}_g$ makes possible the integration of such functions. A set $K \subset M$ is measurable if it is measurable in every chart, and its volume is given by
\[
\textnormal{Vol}_g(K) = \int_K \textnormal{dvol}_g \,.
\]
The $L^p$-space on $M$, denoted as $L^p(M)$, the Sobolev spaces on $M$ (in particular $H^1(M)$) and the distribution space on $M$, $\mathcal{D}'(M)$, are defined in the same way as in the Euclidean space, but using the Riemannian measure, and $g$ as the scalar product on the tangent bundle.  If $\Omega \subset M$, the {\it capacity} of $\Omega$ is given by
\[
\textnormal{cap}_g(\Omega) = \inf\, \, \, \{ \|\nabla^g u\|^2_{H^1} \, \, \, |\, \, \, u \in H^1(M)\, \, \, , \, \, \, u \geq 1\, \, \, a. e.\, \, \,  \textnormal{in a neighborhood of}\,\, \, \Omega\}
\] 
where $a. e.$ means {\it almost everywhere} with respect to the measure $\textnormal{dvol}_g$. The notion of capacity plays an important role in the definition of the space $H^1_0(\Omega)$, see definition \eqref{eq:H10}.

We will denote by $d_g(p,q)$ the distance between two points $p,q \in M$, i.e. the infimum of 
\[
L(\gamma) = \int_a^b \sqrt{g_{\gamma(t)}(\gamma'(t), \gamma'(t))}\, \textnormal{d}t
\]
taken over all continuous and piecewise $C^1$ curve $\gamma: [a,b] \to M$ such that $\gamma(a) = p$ and $\gamma(b) = q$.  For every $p\in M$ and every vector $\nu \in T_pM$ there exists a unique maximal geodesic $\gamma_{p,\nu}$ satisfying $\gamma_{p,\nu} (0)=p$ and $\gamma_{p,\nu}'(0) = \nu$ (here $I$ is an open interval containing the origin, and is maximal with respect to the existence of $\gamma_{p,\nu}$). The injectivity radius at a point $p$ will be
\[
i_p = \inf_{\nu \in \mathcal{S}} c(\nu)
\]
where $\mathcal{S}$ denotes the unit sphere in $T_p M$, and
\[
c(\nu) = \sup \{ t \in I \, \, \, |\, \, \, d(p, \gamma_{p, \nu}(t)) = t\}
\]
The injectivity radius of the manifold is $i_M = \inf_{p} i_p$, that is positive if the manifold is compact. We will denote by $\textnormal{exp}_p: T_p M \to M$ the exponential map, that is given by
\[
\textnormal{exp}_p(\nu) = \gamma_{p,\nu}(1)\,
\]
if $\gamma_{p,\nu}(1)$ exists. We notice that $\gamma_{p,\nu}(t) = \textnormal{exp}_p(t\, \nu)$. If $r$ is less than the injectivity radius at $p$, then we can have a geodesic ball of radius $r$ and center $p$, that in the following will be denoted by
\[
B_r^g(p) = \{ q \in M \,\,\, |\,\,\, d(p,q) \leq r\} = \{ \textnormal{exp}_p (\nu) \, \, \, |\, \, \, \|\nu\|_g < r \} \,.
\]
For vector fields $X , Y, Z$
on $M$ we denote by
\[
R(X, Y)Z = \nabla_Y(\nabla_X Z) - \nabla_X(\nabla_Y Z) - \nabla_{[Y,X]} Z
\]
the Riemann curvature tensor, where $[Y,X]$ is the vector field
\[
[Y,X] = \nabla_Y X - \nabla_X Y\,.
\]
If $\nu_1, \nu_2 \in T_pM$, we will denote by $\mathrm{Rm}_p(\nu_1,\nu_2)$ the sectional curvature at $p$ of the 2-plane determined by $\nu_1, \nu_2$, i.e.
\[
\mathrm{Rm}_p(\nu_1,\nu_2) = \frac{g(R(\nu_1, \nu_2)\nu_1,\nu_2)}{\|\nu_1\|_g\, \|\nu_2\|_g - g(\nu_1, \nu_2)}\,.
\]
On the other hand, we will denote by $\textnormal{Ric}_p$ the Ricci tensor at $p$, i.e.
\[
\textnormal{Ric}_p (\nu_1, \nu_2) = \sum_{i=1}^n g(R(\nu_1, e_1)\nu_2, e_2)
\]
where $\nu_1, \nu_2 \in T_pM$ and $\{e_1,e_2,...,e_n\}$ is an orthonormal basis of $T_p M$. The function $\textnormal{Ric}_p (\nu, \nu)$ on the set of tangent vectors $\nu$ of length 1, that we will write $\textnormal{Ric}_p (\nu)$, is the Ricci curvature at $p$.
If $S$ is an $(n-1)$-submanifold of $M$ and $\nu$ is the normal vector to $S$, we recall that the second fundamental form on $S$ is given by
\[
A(\xi_1, \xi_2) = (\nabla_{\xi_1} \xi_2)^{\nu}
\]
where $\xi_1, \xi_2 \in TS$ and the superscript $\nu$ indicated that we consider only the component on the direction $\nu$. We will denote by $H_S$ the mean curvature of $S$, i.e. the trace of the second fundamental form.

\section{Existence}\label{sect:existence}

\subsection{The compact case}

In this section, $(M,g)$ is a Riemannian manifold.
We start by giving some free boundary formulations of a relaxed version of \eqref{eq:pbforme}, where we replace the volume constraint by an inequality constraint:

\begin{prop}\label{prop:freebdy}
If $\Om^*$ is a solution to 
\begin{equation}\label{eq:pbforme3}
\lambda_1(\Omega^*)=\min\Big\{\lambda_1(\Omega);\;\Omega\, \textrm{quasi-open subset of}\, M, \vol_g(\Omega)\leq m\Big\},
\end{equation}
 then $u_{\Om^*}$ (solution of \eqref{eq:lambda1}) is solution of
\begin{equation}\label{eq:pbfonc0}
\int_{M}\|\nabla^g u_{\Om^*}\|_g^2\, \textnormal{dvol}_g=\min\left\{\int_{M}\|\nabla^g w\|_g^2\, \textnormal{dvol}_g\,,\;w\in H^1(M),\;
\int_M w^2\, \textnormal{dvol}_g=1,\; \vol_g(\Om_{w})\le m \right\},
\end{equation}
and is also solution of
\begin{equation}\label{eq:pbfonc}
J(u_{\Om^*})=\min\left\{J(w), \;w\in H^1(M),\;\textnormal{with }\vol_g(\Om_{w})\le m\right\},
\end{equation}
where we denote $$\displaystyle{J(w):=\int_M\|\nabla^g w\|_g^2\, \textnormal{dvol}_g -\lambda(m) \int_M w^2\, \textnormal{dvol}_g},\;\;\lambda(m):=\inf\Big\{\lambda_1(\Omega);\;\Omega\, \textrm{quasi-open subset of}\, M, \vol_g(\Omega)\leq m\Big\}$$
$$\textnormal{ and }\quad\;\;\Om_{w}=\{w\neq 0\}.$$
Reciprocally, if $u$ solves \eqref{eq:pbfonc0} or \eqref{eq:pbfonc}, then the set $\Om_{u}=\{u\neq 0\}$ is a solution to \eqref{eq:pbforme3}.
\end{prop}

As in Remark \ref{rk:connectedness}, the result fails to be valid if $M$ is disconnected.

\medskip

\begin{proof}
Let $\Om^*$ be a solution to \eqref{eq:pbforme3}. The fact that $u_{\Om^*}$ solves \eqref{eq:pbfonc0} comes easily using the variational formulation \eqref{eq:lambda1}. 
In order to see that $u_{\Om^*}$ is also solution of \eqref{eq:pbfonc}, we apply \eqref{eq:pbfonc0} to $\frac{w}{\|w\|_{2}}$, leading to $J(w)\geq 0$ for all $w\in H^1(M)$ such that $\vol_g(\Om_{w})\leq m $, and we conclude noticing that $J(u_{\Om^*})=0$.

Let now $u$ be a solution to \eqref{eq:pbfonc0} (in the case where $u$ is assumed to solve \eqref{eq:pbfonc}, it is clear that $u$ then also solves \eqref{eq:pbfonc0}). Then given $\Om$ a quasi-open subset of $M$ with $\vol_{g}(\Om)\leq m$, optimizing \eqref{eq:pbfonc0} in $w\in H^1_{0}(\Om)$ such that $\|w\|_{L^2(\Om)}=1$ gives that
$$\lambda_{1}(\Om_{u})=\int_M\|\nabla^g u\|_g^2\, \textnormal{dvol}_g\leq \lambda_{1}(\Om)$$
which concludes the proof.\qed
\end{proof}

The next result deals with the question of saturation of the constraint, and explains how one can link solutions of \eqref{eq:pbfonc} with solutions of \eqref{eq:pbforme2}. This will be used in Section \ref{sect:reg}.
\begin{prop}[Saturation of the constraint]\label{prop:saturation}
We assume $M$ to be connected. If $u$ is a solution to \eqref{eq:pbfonc}, then $\vol_{g}(\Om_{u})=m$. Equivalently, if $\Om^*$ is a solution to \eqref{eq:pbforme3}, then $\vol_{g}(\Om^*)=m$ and $\Om^*=\{u_{\Om^*}\neq 0\}$.
\end{prop}

\begin{proof}
Let $u$ be a solution of \eqref{eq:pbfonc}. Assume to the contrary that $\vol_g(\Om_{u})<m$; then $w=u+t\varphi$ is admissible in \eqref{eq:pbfonc} for any $\varphi$ smooth with small support, and $|t|$ small. It implies
 $$0=\frac{dJ(u+t\varphi)}{dt}_{|t=0}=2\int_{M}g(\nabla^g u,\nabla^g\varphi)\, \textnormal{dvol}_g-2\lambda(m)\int_{M}u\varphi\, \textnormal{dvol}_g$$
which means  $-\Delta_g u=\lambda(m)u$ in $M$ in the sense of distribution. 
Since $u$ has constant sign (see for example \cite[Remark 2.10]{BHP05Lip}), say nonnegative, from strong maximum principle and connectedness of $M$, we get $u>0$ on $M$, which contradicts  $m<\vol_g(M)$  (in fact, in this case $u$ is constant because $\lambda(m)$ would be the first eigenvalue of $M$, which is simple).

If now $\Om^*$ solves \eqref{eq:pbforme3}, then $u_{\Om^*}$ solves \eqref{eq:pbfonc} and has constant sign, say nonnegative. From the strong maximum principle, $u_{\Om^*}>0$ on $\Om^*$, and from the previous point, $\vol(u_{\Om^*})=m$, which concludes the proof.
\qed

\end{proof}

Using Proposition \ref{prop:freebdy}, one can deduce existence of a solution for \eqref{eq:pbforme2}:\\

\noindent{\bf Proof of Theorem \ref{th:existence0}:}
The most convenient formulation for existence is \eqref{eq:pbfonc0}: it is clear that the set of admissible functions is nonempty and that the infimum in \eqref{eq:pbfonc0} is nonnegative. Taking any minimizing sequence $(u_{k})_{k\in\N}$, and using that it is bounded in $H^1(M)$, we get that there exists $u^*\in H^1(M)$ such that, up to a subsequence:
$$\nabla^g u_{k}\rightharpoonup \nabla^g u^*\;\;\textrm{ weakly in }L^2(M),\;\;u_{k}\rightarrow u^*\;\textrm{ a.e. and strongly in }L^2(M),$$
(we use the weak-compactness of closed bounded convex sets in the Hilbert space $H^1(M)$ and the Rellich-Kondrachov Theorem about the compact embedding $H^1(M)\hookrightarrow L^2(M)$ when $M$ is compact, see Theorem 2.34 in \cite{A98Som}). 
We therefore get 
 \[
\vol_g(\Om_{u^*})\leq \liminf_{k\to\infty}\vol_g(\Om_{u_{k}}), \;\;\;\;\;\;\;\int_{M}{u^*}^2\, \textnormal{dvol}_g=1\]
\[\textnormal{and} \qquad \int_{M}\|\nabla^g u^*\|_g^2\, \textnormal{dvol}_g \leq \liminf_{k\to\infty}\int_{M}\|\nabla^g u_{k}\|_g^2\, \textnormal{dvol}_g
\]
so $u^*$ solves \eqref{eq:pbfonc0}, and from Proposition \ref{prop:freebdy}, the set $\Om^*=\{u^*\neq 0\}$ solves \eqref{eq:pbforme3}. In order to build a solution to \eqref{eq:pbforme2}, we simply check that there exists a quasi-open set $\widetilde{\Om^*}\supset\Om^*$ such that $\vol_{g}(\widetilde{\Om^*})=m$. From monotonicity, $\lambda_{1}(\widetilde{\Om^*})\leq \lambda_{1}(\Om^*)$ and therefore $\widetilde{\Om^*}$ solves \eqref{eq:pbforme2}.
\qed

%
%
%

\begin{remark}\rm
Let us notice that it is possible to adapt to the Riemannian setting some more general existence result. Using the method in \cite[Chapter 2, Theorem 2.1]{H17Sha} one can prove existence for
 $$\min\{f(\lambda_{1}(\Om),\ldots,\lambda_{k}(\Om)),\;\;\Om\textrm{ quasi-open }\subset M,\;\vol_g(\Om)= m\},$$
 where $f$ is lower semi-continuous and non-decreasing in each variable, $M$ is compact, and $\lambda_{1}(\Om)\leq \ldots\leq \lambda_{k}(\Om)$ are the $k$ first eigenvalues of the Laplace-Beltrami operator with Dirichlet boundary conditions.
  
Even more generally, it is also possible to adapt the theory of Buttazzo and Dal Maso (see \cite{BD93Ane}) to the case of a compact  Riemannian manifold: one obtains that there exists a solution of the minimization problem:
$$\min\{F(\Om),\;\;\Om\textrm{ quasi-open }\subset M,\;\vol_g(\Om)= m\}$$
 if $M$ is a compact Riemannian manifold, $F$ is a shape functional defined on quasi-open subsets of $M$,  decreasing with respect to set inclusion, and lower-semicontinuous for a suitable topology (namely $\gamma$-convergence, see \cite{BD93Ane} for details).
\end{remark}

\subsection{The non compact case: counter-example to existence}\label{ssect:nonexist}
Through all this subsection, $M$ will be the minimal catenoid in $\mathbb{R}^3$, i.e. the surface given by the equation \begin{equation}\label{catenoid}
x^2+y^2 = \lambda^2 \, \cosh^2\left(\frac{z}{\lambda}\right)
\end{equation}
for some $\lambda > 0$, where $(x,y,z)$ are the coordinates of $\mathbb{R}^3$. It is the surface of revolution generated by a catenary, i.e. the curve 
\[
s \to  \lambda \left(\cosh \left(\frac{s}{\lambda}\right), s \right)\,.
\] 
We consider the natural metric $g$ on $M$, i.e. the metric induced by the immersion of $M$ in $\mathbb{R}^3$. It is useful to express the metric of this surface in the form
\[
\textnormal{d}s^2 = (f(t))^2 \textnormal{d}\theta^2 + \textnormal{d}t^2
\]
for some function $f$. For this we need the length parameterization of the catenary, i.e.
\[
t \to \left( \sqrt{t^2 + \lambda^2} , \lambda \arcsinh \left(\frac{t}{\lambda}\right) \right)\,.
\]
The catenoid is then the image of the annulus $S^1 \times \mathbb{R}$ in $\mathbb{R}^3$ by the application
\[
F(\theta, t) \to \left( \sqrt{t^2 + \lambda^2}\, \cos \theta \, , \, \sqrt{t^2 + \lambda^2}\, \sin \theta \, , \, \lambda\, \arcsinh \left(\frac{t}{\lambda}\right)  \right) 
\]
and the metric is given by
\[
\textnormal{d}s^2 = \left \langle \frac{\partial F}{\partial \theta} , \frac{\partial F}{\partial \theta} \right \rangle \, \textnormal{d}\theta^2 +  \left \langle \frac{\partial F}{\partial t} , \frac{\partial F}{\partial t} \right \rangle \, \textnormal{d}t^2  + 2\left \langle \frac{\partial F}{\partial \theta} , \frac{\partial F}{\partial t} \right \rangle \, \textnormal{d}\theta\, \textnormal{d}t = (t^2 + \lambda^2)\, \textnormal{d}\theta^2 + \textnormal{d}t^2
\]

\medskip

 In this section we prove the following result: 
\begin{theorem}\label{th:nonexist} For any volume $m>0$, problems \eqref{eq:pbforme} or \eqref{eq:pbforme2} have no solution when $M$ is the minimal catenoid defined in \eqref{catenoid}.
\end{theorem}

The proof of Theorem \ref{th:nonexist} follows from the following Propositions \ref{p1} and \ref{p2}. We start by recalling a well known fact about the isoperimetric problem in the minimal catenoid.

\begin{prop}\label{CR} (\cite{CR08The}). 
Let $m > 0$. Then, for all  regular open set $\Omega$ in $M$ whose volume is equal to $m$ we have
\[
P(\Omega) > P(B_m)
\]
where $B_m$ is a Euclidean ball of volume $m$, and $P$ represents the perimeter, i.e. the area of the boundary of the open set.
\end{prop}

\medskip

The Faber-Krahn inequality allows to obtain the same result for the Faber-Krahn profile. 

\begin{prop}\label{p1}
Let $m > 0$. Then, for all  regular open set $\Omega$ in $M$ whose volume is $m$ we have
\[
\lambda_1(\Omega) > \lambda_1(B_m)
\]
where $B_m$ is a Euclidean ball of volume $m$. 
\end{prop}

{\it Proof.} By the Faber-Krahn Theorem (see \cite{C84Eig}, chapter IV, section 2, Theorem 2 and Remark 1), if for all $m>0$ the inequality
\[
P(\Omega) \geq P(B_m)
\]
holds for every  regular open set $\Omega$ 
with equality if and only if $\Omega$ is isometric to $B_m$, then inequality
\begin{equation}\label{inequ:lambda}
\lambda_1(\Omega) \geq \lambda_1(B_m)
\end{equation}
holds for every open set (non necessarily regular) $\Omega$ with equality if and only if $\Omega$ is isometric to $B_m$. The result follows then from Proposition \ref{CR} because if there were in $M$ a regular open set $\Omega$ isometric to $B_m$ then we would have $P(\Omega) = P(B_m)$\hfill $\Box$ 

\medskip

 Note that in the previous proof, the fact that there does not exist in $M$ a regular open set $\Omega$ isometric to $B_m$ follows also from the Egregium Theorem because the Gauss curvature of $M$ is negative). 
\medskip

Now we want to show that in $M$ there does not exist any minimizer of $
\Omega \mapsto \lambda_1(\Omega)
$
under volume constraint. In order to do that, we will show that there exists a minimizing sequence of domains with the same volume whose first eigenvalue converges to the first eigenvalue of a Euclidean ball. We denote by $B_r$ the ball of radius $r$ in $\mathbb{R}^3$, and $e$ will represents the Euclidean metric as usual. We will need the following:


\begin{prop}\label{p2}
For any {$m>0$}, there exists a sequence $B^g_{r_j}(p_j)$, $j \in \mathbb{N}$ such that
\begin{enumerate}
\item ${\rm Vol}_g(B^g_{r_j}(p_j)) = m$ for all $j \in \mathbb{N}$, and
\item $\lambda_1(B^g_{r_j}(p_j)) \to \lambda_1(B_1)$ for $j \to +\infty$. 
\end{enumerate}
\end{prop}

\medskip

\noindent{\bf Proof.} It is convenient to consider just the superior part $M^+$ of the catenoid (i.e. the part with $t>0$). By changing the coordinates as 
\[
t^2 = x^2 + y^2 \qquad \theta = \arctan \frac{y}{x}
\]
(remember that $t>0$), $M^+$ can be seen as $\mathbb{R}^2$ with the metric
\[
g = dx^2 + dy^2 + \frac{\lambda^2}{(x^2 + y^2)^2}\left( xdy-ydx \right)^2
\]
If $p=(x_p, y_p) \in M^+$ and we take a chart centered at $p$, the metric writes 
\[
g = dx^2 + dy^2 + \frac{\lambda^2}{\left[(x+x_p)^2 + (y+y_p)^2\right]^2}\left[ (x+x_p)dy - (y+y_p)dx \right]^2
\]
and it is clear that if $\|p\|_e = \eps^{-1}$, being $\| \cdot \|_e$ the Euclidean norm in $\R^3$, we have
\begin{equation}\label{metr}
g = dx^2 + dy^2 + \mathcal{O}(\eps^2)(dx^2+dxdy+dy^2)
\end{equation}
for $\eps$ small enough, uniformly in a ball $B_r$ for a fixed value $r$.

\medskip

Now let $m>0$. Let $R$ be the radius of the  Euclidean ball of $\mathbb{R}^2$ of volume $m$. We claim that $M^+$ contains a geodesic ball $B^g_{r_q}(q)$ of volume $m$, and for every point $p\in M^+$ with $\|p\|_e \geq \|q\|_e$ there exists $r_p > 0$ such that the geodesic ball $B_{r_p}^g(p)$ is contained in $M^+$ and has volume $m$. Moreover we have that $r_p \to R$ when $\|p\|_e \to +\infty$. \\
{\it Proof of the claim.} If we take a point $q=(x_q,y_q,z_q) \in M^+$ (considering the catenoid in the form \eqref{catenoid}), the intersection of $M^+$ with the vertical cylinder of radius $\sqrt{x_q^2+y_q^2}-1$ and axis $(x,y)=(x_q,y_q)$ is a graph on a 2-dimensional ball of radius $\sqrt{x_q^2+y_q^2}-1$, and its volume tends to $\infty$ when $\|q\|_e \to \infty$. In particular, if $\|q\|_e$ is big enough, it contains a geodesic ball $B^g_{r_q}(q)$ of volume $m$, and this is true also if we replace $q$ by a point $p$ such that $\|p\|_e \geq \|q\|_e$. Now, if we take $\|p\|_e= \eps^{-1}$, it is clear that on a chart centered at $p$ with radius $2R$ we have the estimation \eqref{metr} uniformly in $\eps$, i.e. $g$ converges to the Euclidean metric uniformly in a chart of radius $2R$, and then $r_p \to R$.

\medskip

Now, fix $\eps>0$ small, take $\|p\|_e \in M^+$ such that $\|p\|= \eps^{-1}$ and consider the geodesic ball $B_{r_p}(p)$ of volume $m$. We want to show that the first eigenvalue of the Laplace-Beltrami operator on this ball converges to the first eigenvalue of the Laplacian on $B_R$ when $\eps \to 0$. 
To proceed, it will be more convenient to work on the fixed domain $B_R$, endowed with the metric (depending on $p$)
%
%
\[
\hat g = \left(\frac{r_p}{R}\right)^2 \, g\,.
\]
Let $\phi$ and $\hat \phi$ the first eigenfunction of the Laplace-Beltrami operator on $B_{r_p}^g(p)$ with respect to $g$ and $\hat g$ (normalized to 1 in the $L^2$ norm), and let $\lambda$ and $\hat \lambda$ the associated eigenvalue. We have
\begin{equation}
\label{formula-new}
\left\{
\begin{array}{rcccl}
	\Delta_{\hat{g}} \, \hat \phi + \hat \lambda \, \hat \phi & = & 0 & \textnormal{in} & B_R \\[3mm]
	\hat \phi & = & 0 & \textnormal{on} & \partial B_R
\end{array}
\right.
\end{equation}
being ${\rm Vol}_{\hat g}Ê(B_R) =m$.
When $\eps =0$ the metric $\hat g$ is the Euclidean metric and the solution of (\ref{formula-new}) is therefore given by $\hat \phi = \phi_1$, $\hat \lambda = \lambda_1$, the first eigenfunction and eigenvalue on the Euclidean ball $B_R$. 

\medskip

Let us doing now a formal reasoning.
In $B_R$, for some constant $\mu_\eps$ and some (small) function $w_p$ defined in $B_R$ with 0 Dirichlet boundary condition we have
 \begin{equation}\label{L}
0 = \Delta_{\hat{g}} \, (\phi_1 + w_p) + (\lambda_1 + \mu_p) \, (\phi_1 +w_p)=(\Delta_e + \lambda_1)\,  w_p +  (\Delta_{\hat{g}}-\Delta_e) \, (\phi_1 + w_p) + \mu_p \, (\phi_1 +w_p)
\end{equation}
The kernel of $\Delta_e + \lambda_1$ is spanned by $\phi_1$. Then if we write
 \[
 w_p = w_p^{\|} + w_p^{\bot}
 \]
 where $w_p^{\|} \in \ker (\Delta_e + \lambda_1)$ (say $w_p^{\|} = a_p\, \phi_1$, $a_p \in \mathbb{R}$) and $w_p^{\bot} \in (\ker (\Delta_e + \lambda_1))^{\bot}$, when we project (\ref{L}) onto $\ker (\Delta_e+\lambda_1)$ we obtain 
 \[
(1+a_p)\, \mu_p = \int_{B_R} \phi_1\, (\Delta_e - \Delta_{\hat{g}}) \, (\phi_1 + w_p) \, \textnormal{dvol}_e
\]
It is then natural to ask the function $w_p$ to be in $ (\ker (\Delta_e + \lambda_1))^{\bot}$, in order that $a_p=0$. 
\medskip

According to the previous formal reasoning, for all $w \in \mathcal C^{2, \alpha}Ê(B_1)$  orthogonal to the kernel of $\Delta_e + \lambda_1$, we define the operator
\[
N (\eps, w) : = (\Delta_e +  \lambda_1)\, w +  (\Delta_{\hat g} - \Delta_e+ \mu  )\, (\phi_1 + w)  
\]
where $\mu$ is given by 
\[
\mu  = \mu(\eps, w) = -  \int_{B_R} \phi_1 \,   \left[ (\Delta_{\hat {g}} - \Delta_e ) \, (\phi_1  + w) \right]\, \textnormal{dvol}_e\,.
\]
$N$ is $L^2(B_1)$-orthogonal to $\phi_1$ (with respect to the Euclidean metric). Our aim is to find, for all $\eps$ small enough, a function $w= w(\eps)$ smooth such that $N (\eps, w) = 0$, that is 0 when $\eps=0$. 
In fact, this condition suffices to prove the proposition, because in this case we have 
\[
\lambda_1(B_{r_p}(p)) =  \hat  \lambda = \lambda_1  + \mu
\]
that is a smooth function with respect to $\eps$ and $\mu$ vanishes when $\eps=0$.

\medskip

We have 
\[
N (0,0) =0.
\]
The mapping $N$ is a smooth map from a neighborhood of $(0,0)$ in  $[0, E) \times {\mathcal C}^{2, \alpha}_{\bot,0} (B_1)$ (for some $E>0$) into a neighborhood of $0$ in  ${\mathcal C}_\bot^{0, \alpha} (B_1)$ (here the subscript $\bot$ indicates that functions in the corresponding space are $L^2(B_1)$-orthogonal to $\phi_1$ (for the Euclidean metric) and the subscript $0$ indicates that functions vanish on $\partial B_1$). The differential of $N$ computed at  $(0,0)$, is given by  $\Delta_e + \lambda_1$
since $\hat g $ is the Euclidean metric when $\eps = 0$. Hence the partial differential of $N$ computed at $(0,0)$ is invertible from ${\mathcal C}_{\bot, 0}^{2, \alpha} (B_1)$ into ${\mathcal C}_\bot^{0, \alpha} (B_1)$ and the implicit function theorem ensures, for all $\eps$ small enough (say $\eps < \eps_0$) the existence of a unique solution $w= w(\eps) \in \mathcal C^{2, \alpha}_{\bot , 0} (B_1)$ depending smoothly on $\eps$ such that $N(\eps ,w)=0$. \hfill $\Box$

\section{Regularity}\label{sect:reg}

As annouced in the introduction, we follow the strategy intitiated in \cite{AC81Exi} to study the regularity of free boundaries, and adapt it to solutions of \eqref{eq:pbfonc}. We also adapt the contribution given by G. Weiss in \cite{W98Par} to improve the estimate on the singular set. One major difficulty is due to the volume constraint, for which we need to show that one can consider a penalized version of the problem, similarly to \cite{BL09Reg}. 

Our main goal is to be able to prove existence of blow-up limits and to study them: this will be achieved in Section \ref{ssect:blowup}, only after several preleminary steps:
\begin{itemize}
\item choose an appropriate representative of $u$ solution to \eqref{eq:pbfonc} using properties of subharmonic functions (Section \ref{ssect:apriori}),
\item prove Lipschitz regularity of this representative (Corollary \ref{cor:lip}) which requires a first penalization of the volume constraint (Section \ref{ssect:pen1}). This allows to prove existence of blow-up limits for $u$, see the beginning of the proof of Proposition \ref{prop:blowupexist},
\item prove the nondegeneracy of $u$ near the boundary (Corollary \ref{cor:nondeg}), so that blow-up limits are nontrivial. This requires a deeper analysis of the penalization of the volume constraint, in particular to show that the penalization parameter can be chosen positive for inner perturbations (Sections \ref{ssect:pen} and \ref{ssect:positivity}),
\item deduce from the previous steps that $\Om_{u}$ has finite perimeter (which allows to consider blow-ups near points of the reduced boundary, which have a normal vector in a weak sense) and density estimates (to improve the convergence of blow-up limits), see Sections \ref{ssect:perimeter} and \ref{ssect:density},
\item prove a Weiss type almost monotonicity formula (Proposition \ref{prop:monotonicity}), to show that the blow-up will be a global {\it homogeneous} minimizers of the Alt-Cafarelli functional, see \eqref{eq:pbblowup}.
\end{itemize}
After all these preliminary steps, we can analyse blow-ups and then conclude, using the classification of homogenous minimizers in dimensions less or equal to 5, and the approach by De Silva in \cite{DS11Fre} to study the regularity of flat points, see Section \ref{ssect:conclusion}.

\medskip

In this section, $M$ denotes a smooth, connected and compact Riemannian manifold. The compactness hypothesis is made only for convenience, all arguments can be localized, which explains that we drop this hypothesis in the statement of Theorem \ref{th:reg0}. The function $u$ will denote a solution to \eqref{eq:pbfonc}, as we know from Proposition \ref{prop:saturation} that $\Om^*$ solution to \eqref{eq:pbforme2} can be seen as $\Om^*=\Om_{u}=\{u\neq0\}$ where $u$ solves \eqref{eq:pbfonc}. It is also well-known that $u$ can be assumed to be nonnegative, so that $\Om_{u}=\{u>0\}$.

\medskip

Let $x_0\in M$, let $E_{1} ,\ldots, E_n$ be an orthonormal basis of the tangent space $T_{x_0} M$. We denote by $\mbox{exp}_{x_0}$ the exponential map on the manifold $M$ at the point $x_0$. The coordinates we want to use are the classical geodesic normal coordinates at $x_0$, 
\[
x : =(x^1, \ldots, x^n) \in \mathbb R^n\ ,
\]
defined by the chart
\[
X (x) : =\mbox{exp}_{x_0} \left( \Theta (x) \right) \,
\]
where
\begin{equation}\label{eq:theta}
\Theta (x) : = \sum_{i=1}^n x^i \, E_i \in T_{x_0} M \, .
\end{equation}
In order to study the regularity of our optimal set $\Omega^*$, we fix an arbitrary point $x_0 \in \partial \Omega^*$, and we just consider a neighborhood of $x_0$ in $M$ (we do not need the rest of the manifold). If we fix $r_0$ as a positive constant less than the cut locus at $x_0$ and we define the geodesic ball of radius $r_0$ and center $x_0$ as
\[
B_{r_0}^g (x_0) : =  \left\{ \mbox{Exp}_{x_0} (  \Theta (x)) \qquad : \quad  x \in \mathbb R^{n} \qquad 0 \leq |x|   < r_0 \right\} \, ,
\]
we just need to understand the behavior of $\partial \Omega^* \cap B_{r_0}^g (x_0)$. 
If $R$ is the Riemann curvature tensor on the manifold $M$, we have a Taylor expansion of the coefficients $g_{ij}(x)$ of the metric in these geodesic normal coordinates given by
\begin{equation}
g_{ij} = \delta_{ij} + \frac{1}{3} \, R_{ikj\ell} \, x^{k} \, x^{\ell} + {\mathcal O}(|x|^{3}),
\end{equation}
where $\delta_{ij}$ is the Kronecker symbol and $R_{ikj\ell} = g\big( R(E_{i}, E_{k}) \, E_{j} ,E_{\ell}\big)$ at the point $x_0$ (here the Einstein summation convention is understood). 
%
%
Hence we can choose $r_0$ small enough such that we have 
$$\frac{id}{2}\leq g\leq 2id\, , $$
in $B_{r_0}^g(x_0)$, and so the Laplace-Beltrami operator is uniformly elliptic in $B_{r_0}^g(x_0)$. Moreover, a straightforward computation shows that at the point of coodinates $x$
\[
\begin{array}{rllll}
g^{ij}(x) & = & \displaystyle \delta_{ij} - \frac{1}{3} \, R_{ikj\ell} \, x^{k} \, x^{\ell}  + {\mathcal O} (|x|^{3})\\[3mm]
\log | g| (x)& = & \displaystyle \frac{1}{3} \,R_{k\ell} \, x^{k} \, x^{\ell}  + {\mathcal O} (|x|^{3})
\end{array}
\]
where
$$
R_{k\ell} = \sum_{i=1}^{n} R_{iki\ell}\,,
$$
and then for a function $f$ with bounded fist derivatives we have
\[
\|\nabla^{g} f - \nabla^e f \|_{\infty}\leq C\, r_0\, ,
\]
in $B_{r_0}^{g}(x_0)$, where the constant $C$ depends on the bound of the first derivatives of the function $f$ in $B_{r_0}^{g}(x_0)$, and for a function $f$ with bounded first and second derivatives we have
\begin{equation}\label{eq:deltag}
\| \Delta_{g} f - \Delta_e f\|_{\infty} \leq C\, r_0\, 
\end{equation}
in $B_{r_0}^{g}(x_0)$, where the constant $C$ depends on the bound of first and second derivatives of the function $f$ in $B_{r_0}^{g}(x_0)$. 

\medskip

In some parts of the proof of the regularity result, it will be necessary to study the behavior of a function in a very small geodesic ball $B_\eps^g(x_0)$, and for that aim we will do a {\it scaling}. In the Euclidean framework, this means that given a function $w(x)$ for $x \in B_\eps$ we define the function $\tilde w(y) = w(\eps\, y)$ for $y \in B_1$, where $B_\eps$ and $B_1$ are the Euclidean balls of radius $\eps$ and 1. This means that the metric on $B_1$ (i.e. the Euclidean metric) is not the metric $g$ induced on $B_1$ by the parameterization $x=\eps y$ of $B_\eps$, but $\eps^{-2}\, g$. We do the same in the Riemannian framework. Given $\eps>0$ small enough, we define on the manifold $M$ the metric $\bar g : =  \eps^{-2} \, g$ and the parameterization given by
\[
Y ( y) : =\mbox{exp}_{x_0}^{g} \left(\eps \,  \Theta (y) \right)
\]
with $y = (y_1,...,y_n) \in B_R$, being $B_R$ the Euclidean ball of radius $R$, and $\Theta$ defined in \eqref{eq:theta}. 
If $x$ are the normal geodesic coordinates around $x_0$ of the same point as $y$, it is clear that 
\[
\bar g_{ij}(y)\, \textnormal{d}y^i \,  \textnormal{d}y^j =  \bar g = \eps^{-2} \, g = \eps^{-2} \, g_{ij}(x)\, \textnormal{d}x^i \,  \textnormal{d}x^j = \eps^{-2} \, g_{ij}(\eps \,  y)\, \eps^2 \, \textnormal{d}y^i \,  \textnormal{d}y^j = g_{ij}(\eps \,  y)\, \textnormal{d}y^i \,  \textnormal{d}y^j 
\]
Hence in the coordinates $y$, the metric $\bar g$ is given by
\[
\bar g_{ij}(y) = g_{ij}(\eps \, y)
\]
i.e.
\[
\bar g_{ij}(y) =  \delta_{ij} + \frac{1}{3} \, \eps^2\, R_{ikj\ell} \, y^{k} \, y^{\ell} + {\mathcal O}(\eps^{3})\,.
\]

\bigskip

\subsection{A priori regularity}\label{ssect:apriori}

In this first paragraph, we recall standard properties of eigenfunctions in a domain, without using the optimality of the shape itself, but only the optimality of $u$ in \eqref{eq:lambda1}. Therefore in this section, $\Om\subset M$ denotes a bounded quasi-open set. 

\begin{prop}\label{prop:reg0} Let $u$ a nonnegative solution to \eqref{eq:lambda1}. Then
\begin{itemize}
\item $\Delta_{g} u+\lambda_{1}(\Omega)u\geq 0$ in $\D'(M)$, and in particular $\Delta_{g} u$ is a signed Radon measure on $M$,
 \item $u$ is 
  bounded in $M$.
\end{itemize}
\end{prop}

\begin{proof}
We know that
\begin{equation}\label{eq:basicEL}
\forall v\in H^1_{0}(\Om), \int_M {g}(\nabla^{g} u,\nabla^{g} v)\,  \textnormal{dvol}_{g}=\lambda_{1}(\Omega)\int_M uv \, \textnormal{dvol}_{g}\,.
\end{equation}
Let $\varphi\in C^\infty_{c}(M)$ such that $\varphi\geq 0$. We introduce $p_{n}(s)=\inf\{ns,1\}_{+}$ for $s\in\R$ and apply \eqref{eq:basicEL} to $v=\varphi p_{n}(u)$ which belongs to $H^1_{0}(\Om)$. This gives
$$\int_M {g}\left(\nabla^{g} u,\nabla^{g} \varphi\right) p_{n}(u) \,  \textnormal{dvol}_{g} +\underbrace{\int_M \|\nabla^{g} u\|_g^2\,\varphi\, p_{n}'(u) \,  \textnormal{dvol}_{g}}_{\geq 0}=\lambda_{1}(\Omega)\int_M \left(u\varphi\right) p_{n}(u)\,  \textnormal{dvol}_{g}\,.$$
We make $n$ go to $\infty$ and use that $p_{n}(u)$ converges to $\mathbbm{1}_{\{u>0\}}$; using that $u\geq 0$ in $M$, we obtain
$$\forall \varphi\in\C^\infty_{c}(M)\textrm{ such that }\varphi\geq 0, \;\;\;\int_M g(\nabla^{g}  u,\nabla^{g}  \varphi)\,  \textnormal{dvol}_{g} -\lambda_{1}(\Omega)\int_M u\varphi \,  \textnormal{dvol}_{g}\leq 0,$$
which leads to $\Delta_{g} u+\lambda_{1}(\Omega)u\geq 0$ in the sense of distribution on $M$.
The second point follows classically from the fact that
$\left[-\Delta_{g}  -\lambda_{1}(\Omega)\right]u\leq 0$ and $u\in L^2(M)$, see for example \cite[Theorem 8.15 and p. 214]{GT01Ell}.
\qed
\end{proof}

\begin{remark}\label{rk:continuity} We recall \cite[Th 1.1]{BBL17Mea} (see also \cite{BH15The}) that generalizes the mean value property in the Riemannian context:
 for every $x\in M$, there exists a family $\{D_r(x)\}_{0<r<r_0}$ of open sets, monotone with respect to the inclusion, and converging to the point $x$ when $r \to 0$ such that for any $v$ subharmonic in $M$ (meaning that $\Delta_{g} v\geq 0$), 
 we have that 
$r\in(0,r_{0})\mapsto \displaystyle{\fint_{D_{r}(x)}v \,\textnormal{dvol}_{g}}$ is increasing. Moreover, for almost every $x\in M$ we have
$$v(x)= \lim_{r\to 0}\fint_{D_{r}(x)}v \,\textnormal{dvol}_{g},$$
and this representative is defined everywhere on $M$ and is upper-semi-continuous.
\end{remark}

We add the following lemma\footnote{We didn't find a reference for this result, but it can be found in \url{https://cuhkmath.wordpress.com/2015/08/14/mean-value-theorems-for-harmonic-functions-on-riemannian-manifolds/}}:
\begin{lemma}\label{lem:ricci}
Let $v:M\to\R$, $x_{0}\in M$ and $r_{0}>0$ smaller than the injectivity radius.
\begin{itemize}
\item 
In $B_{r_{0}}^{g}(x_{0})$, if 
${\Delta_g v\le 0}$, $v\ge0$ and the Ricci curvature satisfies ${\textnormal{Ric}\ge (n-1)k}$ for some constant $k\in\R$,  then 
$r\in(0,r_{0})\mapsto \displaystyle{ \frac{1}{\vol_{g_k}(S^{g_k}_r)} \int_{S^g_r(x_0)} v\, \textnormal{dvol}_g|_{S^g_r(x_0)}}$
is decreasing,
where $S^g_r(x_0)=\partial B_{r}^g(x_{0})$ and ${S^{g_k}_r}$ is a geodesic sphere of radius ${r}$ in the space form with metric $g_k$ with constant sectional curvature $k$. 
\item  
In $B_{r_{0}}^{g}(x_{0})$, if  ${\Delta_g v\ge 0}$, $v\ge 0$ and the sectional curvature satisfies ${\mathrm{Rm}\le k}$ for some constant $k$,  then 
$r\in(0,r_{0})\mapsto \displaystyle \frac{1}{\vol_{g_k}(S^{g_k}_r)} \int_{S^g_r(x_0)} v\, \textnormal{dvol}_g|_{S^g_r(x_0)}$
is increasing.

\end{itemize}
 Here $\vol_{g_k}(S^{g_k}_r)$ represent the area of $S^{g_k}_r$ in the metric induced by $g_k$.
\end{lemma}
\begin{proof} 
For the first case, we assume $v$ smooth and apply Hadamard formula for derivation of integrals defined in a domain with variable boundary:
\[
\begin{array}{rl} \displaystyle 0\ge\int_{B^g_r(x_0)} \Delta_g v \, \textnormal{dvol}_g =&\displaystyle \int_{S^g_r(x_0)}g( \nabla^g v, \nu )\, \textnormal{dvol}_g|_{S^g_r(x_0)}\\ 
=&\displaystyle \frac{d}{dr}\left(\int _{S^g_r(x_0)}v \, \textnormal{dvol}_g|_{S^g_r(x_0)}\right)-  \int_{S^g_r(x_0)}v H_{S^g_r(x_0)}\, \textnormal{dvol}_g|_{S^g_r(x_0)}\end{array}
\]
 where $H_{S^g_r(x_0)}$ represents the mean curvature at points of $S^g_r(x_0)$. Let $H_{S^{g_k}_r}$ be the mean curvature of a geodesic sphere of radius $r$ in the space form of constant curvature $k$, i.e. 
\[
H_{S^{g_k}_r} = (n-1)\frac{c_k(r)}{s_k(r)}
\]
where
\[{s_k(r)= \begin{cases} r\quad &\textrm{if }k=0\\ \sin (\sqrt k r)&\textrm{if }k>0\\ \sinh (\sqrt {-k} r)&\textrm{if }k<0 \end{cases} } \qquad \textnormal{and} \qquad {c_k(r)= \begin{cases} 1\quad &\textrm{if }k=0\\ \cos (\sqrt k r)&\textrm{if }k>0\\ \cosh (\sqrt {-k} r)&\textrm{if }k<0 \end{cases} }\, .
\] 
If ${\textnormal{Ric}\ge (n-1)k}$, by the mean curvature comparison theorem (see \cite{Zhu}) we have $H \leq H_{S^{g_k}_r}$, and then
\[
\begin{array}{rl} \displaystyle 0\ge&\displaystyle \frac{d}{dr}\left(\int _{S^g_r(x_0)}v \, \textnormal{dvol}_g|_{S^g_r(x_0)}\right)- \int_{S^g_r(x_0)}v H_{S^{g_k}_r} \, \textnormal{dvol}_g|_{S^g_r(x_0)}\\ =&\displaystyle \frac{d}{dr}\left(\int _{S^g_r(x_0)}v \, \textnormal{dvol}_g|_{S^g_r(x_0)}\right)- (n-1)\frac{c_k(r)}{s_k(r)}\int_{S^g_r(x_0)}v\, \textnormal{dvol}_g|_{S^g_r(x_0)}. \end{array}
\]
If $w_n$ is the volume of the unit $n$-dimensional ball, we obtain
\[
0\ge w_n s_k(r)^{n-1} \frac{d}{dr} \left(\int_{S^g_r(x_0)} v\, \textnormal{dvol}_g|_{S^g_r(x_0)}\right) -(n-1)w_{n}s_k(r)^{n-2}c_k(r) \int_{S_r^g(x_{0})} v\, \textnormal{dvol}_g|_{S^g_r(x_0)}\, ,
\]
which implies
\[
0\ge \frac{d}{dr} \left(\frac{\displaystyle \int_{S^g_r(x_0)}v \, \textnormal{dvol}_g|_{S^g_r(x_0)}}{w_n s_k(r)^{n-1}}\right) =\frac{d}{dr}\left(\frac{1}{\vol_{g_k}(S^{g_k}_r)}\int _{S^g_r(x_0)} v\, \textnormal{dvol}_g|_{S^g_r(x_0)}\right)
\]
and leads to the result. The case of non-smooth function $v$ is obtained by approximation, because smooth subharmonic functions are dense in the space of subharmonic functions (see for example \cite{BonLan}).

The proof for the subharmonic case is similar except that we require the sectional curvature to be less or equal to $k$ in order to use the Hessian comparison theorem.\qed

\end{proof}


\begin{corollary}\label{cor:continuity}
Let $u$ be a solution of \eqref{eq:lambda1}. Then $u$ can be pointwise defined by an upper-semi-continuous representative (still denoted $u$) with the following formula (see Remark \ref{rk:continuity} for the definition of $D_{r}(x)$):
\begin{equation}\label{eq:pointwise}u(x)= \lim_{r\to 0}\fint_{D_{r}(x)}u \,\textnormal{dvol}_{g}.
\end{equation}
Moreover, there exist $c,c'>0$ and $r_{0}>0$ such that 
for any $r<r_{0}$, one has
\begin{equation}\label{eq:meanvalue}
u(x)\leq c\left(\fint_{B_{r}^{g}(x)}u\, \textnormal{dvol}_{g}+r^2\right)\leq c'\left(\fint_{S_{r}^{g}(x)}u\, \textnormal{dvol}_{g}|_{S^{g}_r(x_0)}+r^2\right).
\end{equation}
\end{corollary}
\begin{proof}
From Proposition \ref{prop:reg0}, we know that $\Delta_{g} u\geq -\lambda\|u\|_{\infty}$. Let fix $x_{0}\in M$ and $r_{0}$ such that $B^{g}_{r_{0}}(x_{0})$ can be represented by a unique chart, and define $w(x)=|x|^2$ in local geodesic coordinates. Then in $B^{g}_{r}(x_{0})$ for $r\leq r_{0}$, one has using \eqref{eq:deltag},
$$\Delta_{g}w\geq \Delta w-C r$$
 for some constant $C$ independant on $r$.
Therefore $\widetilde{u}:=(u+\lambda\|u\|_{\infty}w)$ is sub-harmonic in $B^{g}_{r_{1}}(x_{0})$ if $r_{1}$ is small enough. 
Applying Remark \ref{rk:continuity}, we obtain first
that 
\[
\lim_{r\to0}\fint_{D_{r}(x_{0})}(u+\lambda\|u\|_{\infty}w)\, \textnormal{dvol}_{g}=\lim_{r\to0}\fint_{D_{r}(x_{0})}u\, \textnormal{dvol}_{g}
\] 
exists.
As $g$ is uniformly elliptic inside $B^{g}_{r_{1}}(x_{0})$, from \cite[Theorem 6.3]{BH15The}, one has 
\[
B^{g}_{c_{0}r}(x_{0})\subset D_{r}(x_{0})\subset B^{g}_{c_{1}r}(x_{0})
\] 
for some constants $0<c_{0}<c_{1}$, and $r<\frac{r_{1}}{c_{1}}$.
From Lebesgue differentiation Theorem, \eqref{eq:pointwise} is valid almost everywhere, and this representative is upper semi-continuous, see also \cite{BH15The}. 
Using the representative of $u$ given in \eqref{eq:pointwise} and the monotonicity of 
 \[
 r\in(0,r_{1})\mapsto\int_{D_{r}(x_{0})}\widetilde{u}\,\textnormal{dvol}_{g},
 \] 
 we obtain 
 $$u(x_{0})\leq \fint_{D_{r/c_{1}}(x_{0})}\widetilde{u}\, \textnormal{dvol}_{g}
 \leq \frac{1}{\vol_g(B^{g}_{c_{0}r/c_{1}}(x_{0}))}\int_{B^{g}_{r}(x_{0})}\widetilde{u} \, \textnormal{dvol}_{g} 
 = \frac{\vol_g(B^{g}_{r}(x_{0}))}{\vol_g(B^{g}_{c_{0}r/c_{1}}(x_{0}))}\fint_{B^{g}_{r}(x_{0})}\widetilde{u} \, \textnormal{dvol}_{g}
 $$
$$ \leq 2\left(\frac{c_{1}}{c_{0}}\right)^n\fint_{B^{g}_{r}(x_{0})}\widetilde{u}\, \textnormal{dvol}_{g}\leq c\left(\fint_{B^{g}_{r}(x_{0})}u\, \textnormal{dvol}_{g}+r^2\right),$$
where $c=2\max\{1,\lambda\|u\|_{\infty}\}\displaystyle{\left(\frac{c_{1}}{c_{0}}\right)^n}$.
For the last property, we show there exists $c_{2}$ such that
$$\fint_{B^{g}_{r}(x_{0})}\widetilde{u}\, \textnormal{dvol}_{g} \leq c_{2}\fint_{S^{g}_{r}(x_{0})}\widetilde{u}\, \textnormal{dvol}_{g}|_{S^{g}_r(x_0)}.$$
First, we choose $k$ positive such that ${\mathrm{Rm}\le k}$ in the ball $B^{g}_{r}(x_{0})$. Therefore, from Lemma \ref{lem:ricci}, we know that 
\[
r\in(0,r_{1})\mapsto \displaystyle \frac{1}{\vol_{g_k}(S^{g_k}_r)} \int_{S^g_r(x_0)} \widetilde{u}\, \textnormal{dvol}_g|_{S^g_r(x_0)}
\] 
is increasing. Moreover
\begin{eqnarray*}\fint_{B^{g}_{r}(x_{0})}\widetilde{u}\, \textnormal{dvol}_{g}&=&\frac{1}{\vol_g(B^{g}_{r}(x_{0}))}\int_{0}^r\left(\int_{S^{g}_{s}(x_{0})}\widetilde{u}\, \textnormal{dvol}_{g}|_{S^{g}_s(x_0)}\right)ds\\
&=&\frac{1}{\vol_g(B^{g}_{r}(x_{0}))}\int_{0}^r\left(\frac{1}{\vol_{g_k}(S^{g_{k}}_{s}(x_{0}))}\int_{S^{g}_{s}(x_{0})}\widetilde{u}\, \textnormal{dvol}_{g}|_{S^{g}_s(x_0)}\right)\vol_{g_k}(S^{g_{k}}_{s}(x_{0}))ds\\
&\leq& \frac{\vol_{g_k}(B^{g_{k}}_{r}(x_{0}))}{\vol_g(B^{g}_{r}(x_{0}))}\left(\frac{1}{\vol_{g_k}(S^{g_{k}}_{r}(x_{0}))}\int_{S^{g}_{r}(x_{0})}\widetilde{u}\, \textnormal{dvol}_{g}|_{S^{g}_r(x_0)}\right)
\end{eqnarray*}
and hence the result, as $\displaystyle{\frac{\vol_{g_k}(B^{g_{k}}_{r}(x_{0}))}{\vol_g(B^{g}_{r}(x_{0}))}}$ is bounded from above for $r\leq r_{1}$.
\qed
\end{proof}

\subsection{First penalization}\label{ssect:pen1}

In the first steps of the analysis in \cite{AC81Exi}, which will lead to the optimal regularity for $u$ (see Section \ref{ssect:lip}), the authors use a test function whose support is larger than $\Om_{u}=\{u\neq 0\}$. Therefore, this test function is a priori not allowed because of the volume constraint. We therefore prove that \eqref{eq:pbfonc} is equivalent to a penalized version. Note that one could skip this step as in Section \ref{ssect:pen} we will prove a more precise penalization result (see \cite{RTV18Exi} where this strategy is applied), but we decided to keep the result in this section as they are more straightforward and global. We use in this section that $M$ is compact. Nevetheless this hypothesis is not required for Theorem \ref{th:reg0} since this section can be skipped.


\begin{prop}[Global Penalization from above]\label{prop:pen1}
If $u$ is a solution of \eqref{eq:pbfonc}, 
then there exists $\mu^*\in\R_{+}$ such that for all  $v\in H^1(M)$
\begin{equation}\label{eq:pbpen1}
\int_M \|\nabla^{g}  u\|_g^2 \,  \textnormal{dvol}_{g}\leq \int_M \|\nabla^{g}  v\|_g^2\,  \textnormal{dvol}_{g}+\lambda(m)\left[1-\int_M v^2\,  \textnormal{dvol}_{g}\right]^++\mu^* \Big[\vol_g(\Om_{v})- m\Big]^+.
\end{equation}
\end{prop}

We will argue as in \cite[Theorem 2.4]{BHP05Lip} where they deal with the Euclidian case. Only the last part of the proof requires some new developments in the Riemannian case, though we will recall the whole proof. 
 
 \begin{remark}\label{rk:pen}
 This implies in particular that if $\Om^*$ solves \eqref{eq:pbforme2} and $M$ is compact, then for $\mu$ large enough, 
 $$\forall \Om\subset M, \;\;\;\;\lambda_{1}(\Om^*)\leq \lambda_{1}(\Om)+ \mu\left[\vol_g(\Om)-m\right]^+$$
 \end{remark}

 \begin{proof}
Let us 
introduce $u_{\mu}$ a solution of 
 $$\min\left\{G_{\mu}(v):=\int_M \|\nabla^{g}  v\|_g^2 \,  \textnormal{dvol}_{g}+\lambda_{ m}\left[1-\int_M  v^2\,  \textnormal{dvol}_{g}\right]^++\mu \Big[\vol_g(\Om_{v})- m\Big]^+, v\in H^1(M)\right\}.$$
The existence of such $u_{\mu}$ follows from standard compactness arguments: observing indeed that $G_{\mu}(|w|)= G_{\mu}(w)$,
 and that $G_{\mu}(w/\|w\|_{L^2})\leq G_{\mu}(w)$ if $\|w\|_{L^2}\geq 1$, 
one can consider a minimizing sequence in the set 
\[
\{w\in H^1(M)\;/\; w\geq 0, \|w\|_{L^2}\leq 1\}\,.
\] 
Such minimizing sequence has a gradient uniformly bounded in $L^2$ and is therefore bounded in $H^1(M)$; similarly than in the proof of Theorem \ref{th:existence0}, this provides a solution $u_{\mu}$, which is moreover non-negative and such that $\|u_{\mu}\|_{L^2}\leq 1$.
Then, it only remains to show that for $\mu$ large enough, we necessarily have $\vol_g(\Om_{u_{\mu}})\leq m$. Indeed, in that case we have, using optimality of $u_{\mu}$, $G_{\mu}(u_{\mu})\leq G_{\mu}(u)$ and one hand, and on the other hand
 using \eqref{eq:pbfonc}: 
 \begin{equation*}
 \int_M \|\nabla^{g}  u\|_g^2\,  \textnormal{dvol}_{g}\leq
 \int_M \|\nabla^{g}  u_{\mu}\|_g^2\,  \textnormal{dvol}_{g}+\lambda(m)\left[1-\int_M u_{\mu}^2\,  \textnormal{dvol}_{g}\right]
  \leq G_{\mu}(u_{\mu})
 \end{equation*}
 so $u$ is also a minimizer for \eqref{eq:pbpen1}.
 
 \medskip
 
 We therefore assume that $\vol_g(\Om_{u_{\mu}})>m$; then we can write, $G_{\mu}(u_{\mu})\leq G_{\mu}((u_{\mu}-t)^{+})$ for $t>0$ small enough and get, using in particular that $\|(u-t)^+\|_{L^2}\leq \|u\|_{L^2}\leq 1$:
\begin{eqnarray*} \int_{\{0<u_{\mu}<t\}} \|\nabla^{g}  u_{\mu}\|_g^2\,  \textnormal{dvol}_{g}+ \mu\, \vol_g(\{0<u_{\mu}<t\})&\leq&  
  \lambda_{ m}\left[\int_M u_{\mu}^2\,  \textnormal{dvol}_{g}-\int_{\{u_{\mu}\geq t\}}(u_{\mu}-t)^2\,  \textnormal{dvol}_{g}\right]\\&\leq& \lambda_{ m}\left[\int_{\{0<u_{\mu}<t\}} u_{\mu}^2\,  \textnormal{dvol}_{g}+\int_{\{u_{\mu}\geq t\}}2tu_{\mu}\,  \textnormal{dvol}_{g}\right] 
 \end{eqnarray*}
and using the Cauchy-Schwarz inequality and the fact that $\|u_{\mu}\|_{L^2}\leq 1$, we get:
$$ \int_{\{0<u_{\mu}<t\}}\left( \|\nabla^{g}  u_{\mu}\|_g^2-\lambda_{ m}u_{\mu}^2\right)\,  \textnormal{dvol}_{g}+\mu\, \vol_g(\{0<u_{\mu}<t\}) \leq 2\lambda_{ m}t \, \vol_g(\Om_{u_{\mu}})^{1/2}.$$ 
With the co-area formula, we obtain
 $$\int_0^t\int_{\{u_{\mu}=s\}}\underbrace{\left[ \|\nabla^{g}  u_{\mu}\|_g+\frac{\mu-\lambda_{ m}u_{\mu}^2}{\|\nabla^{g}  u_{\mu}\|_g}\right]}_{\geq \sqrt{2\mu}\;\textrm{for }s \textrm{ such that }\lambda_{ m}s^2\leq \mu/2}\,  \textnormal{dvol}_{g|_{\{u_{\mu}=s\}}} dt\leq 2\lambda_{ m}t \, \vol_g(\Om_{u_{\mu}})^{1/2},\;\;\;\textrm{ for }t\textrm{ small enough}$$
 which gives, dividing by $t$ and letting $t\to 0$:
 $$\sqrt{2\mu} \, P(\Om_{u_{\mu}})\leq 2\lambda_{ m} \vol_g(\Om_{u_{\mu}})^{1/2}, \;\;\;\;\textrm{ and then }\;\;\;\;\;\sqrt{2\mu}\leq \frac{2\lambda(m)\vol_g(M)^{1/2}}{I_{M}(\vol_g(\Om_{u_{\mu}}))},$$
where $I_{M}:[0,\vol_g(M)]\to\R_{+}$ is the isoperimetric profile of $M$, which is positive on $(0,\vol_g(M))$ as $M$ is compact and connected.
From the estimate
$$G_{\mu}(u_{\mu})\leq G_{\mu}(u)=\int_M \|\nabla^{g}  u\|_g^2 \,  \textnormal{dvol}_{g}\,, \;\;$$
we know that $\mu \big[\vol(\Om_{u_{\mu}})-m\big]$ is bounded uniformly in $\mu$, so there exists $\mu_{0}$ such that forall $\mu\geq \mu_{0}$, $\vol_g(\Om_{u_{\mu}})\in \left[m,\frac{\vol_g(M)+m}{2}\right]$.
As a conclusion, if $\mu> \displaystyle{\max\left\{\mu_{0}, \frac{2\lambda(m)\vol_g(M)}{\inf\left\{I_{M}(s), s\in\left[m,\frac{\vol_g(M)+m}{2}\right]\right\}}\right\}}$, then $\vol_g(\Om_{\mu})\leq m$, which concludes the proof.
 \qed
 \end{proof}

 \subsection{Lipschitz continuity of the first eigenfunction}\label{ssect:lip}

A first step in the regularity theory is to study the regularity of the state function, seen as a function defined on $M$. The regularity is obvious inside or outside $\Om_{u}$ (so far though, we do not know yet that the interior of $\Om_{u}$ is not empty), but the regularity across the free boundary is non trivial, especially as we don't know anything about the regularity of this free boundary. It is clear that one cannot expect more than Lipschitz continuity: even if we already knew that $\Om_{u}$ is smooth, then the eigenfunction vanishes outside $\Om_{u}$ and has  a linear growth from the boundary, inside $\Om_{u}$. The purpose of this section is to prove the Lipschitz continuity of $u$ despite the lack of knowledge about $\partial\Om_{u}$; this is often referred to as the {\it optimal regularity of the state function}. This will have some consequences about weak regularity of the free boundary, and will be useful to prove existence of blow-ups. 
 We will use here the penalization result of the previous section. But we could also use the refined penalization result of Section \ref{ssect:pen} (for which the compactness of the manifold is not required), see also \cite{RTV18Exi}.



As in \cite{AC81Exi}, we express the Lipschitz regularity as an uniform bound of the mean value of $u$ on spheres crossing $\partial\Om_{u}$, so the main tool in this section is the following lemma:
\begin{lemma}[Upper bound to the growth of $u$ near the boundary]\label{lem:lip}
Let $u$ be a solution of \eqref{eq:pbfonc}. There exist $C>0$ and $r_{0}>0$ such that,
for all geodesic ball $B_r^{g}(x_0)\subset M$ with $r\leq r_{0}$,
\begin{equation}
\;\frac{1}{r}\fint_{\partial B_r^{g}(x_0)}u \,  \textnormal{dvol}_{g|_{\partial B_r^{g}(x_0)}} \geq C
 \;\;\;\Longrightarrow\;\;\;  u>0 \textrm{ on } B_r^{g}(x_0).
\end{equation}
\end{lemma}
\begin{proof}
Let $B_r^{g}(x_0)\subset M$ and $v$ satisfying 
\begin{equation}\label{eq:v1}
\left\{\begin{array}{cccl}
-\Delta_{g} v & = &\lambda(m) u &\textrm{ in } B_r^{g}(x_0),\\
v & = & u &\textrm{ outside } B_r^{g}(x_0).
\end{array}
\right.
\end{equation}
We first notice that from maximum principle $v\geq u\geq 0$, so $\Om_{u}\subset\Om_{v}$.

\noindent\textbf{Step 1: } Using optimality of $u$ for the penalized version given in Proposition \ref{prop:pen1}, let us first prove 
\begin{equation}\label{eq:estimate1}
\int_{B_r^{g}(x_0)} \|\nabla^{g} (u-v)\|_g^2\, \textnormal{dvol}_{g} \leq \mu^*\, \vol_g(\{u=0\}\cap B_r^{g}(x_0)).
\end{equation}
We compare the energies of $u$ and $v$ in \eqref{eq:pbpen1}, which gives,
$$J(u)-J(v)\leq \mu^*\Big[\vol_g(\Om_{v})-\vol_g(\Om_{u})\Big]\,.$$
We now notice, using that $u-v$ vanishes on $\partial B_r^{g}(x_0)$,
\begin{eqnarray}\label{moyp2}
J(u)-J(v)
& = &  \int_{B_r^{g}(x_0)} \|\nabla^{g} (u-v)\|_g^2  \,  \textnormal{dvol}_{g}+2\int_{B_r^{g}(x_0)} g(\nabla^{g} (u-v),\nabla^{g} v)  \,  \textnormal{dvol}_{g}-\lambda(m)\int_{B_r^{g}(x_0)}(u^2-v^2) \,  \textnormal{dvol}_{g}
 \nonumber \\
 & = &\int_{B_r^{g}(x_0)} \|\nabla^{g} (u-v)\|_g^2 \,  \textnormal{dvol}_{g}-2\int_{B_r^{g}(x_0)}\Delta_{g} v\, (u-v) \,  \textnormal{dvol}_{g}-\lambda(m)\int_{B_r^{g}(x_0)}(u^2-v^2)  \,  \textnormal{dvol}_{g}\nonumber \\
 &=&\int_{B_r^{g}(x_0)} \|\nabla^{g} (u-v) \|_g^2 \,  \textnormal{dvol}_{g}+\lambda(m)\int_{B_r^{g}(x_0)}(u-v)^2 \,  \textnormal{dvol}_{g}
\end{eqnarray}
and therefore, 
\begin{equation}\label{eq:estimate(u-v)}
\int_{B_r^{g}(x_0)} \|\nabla^{g} (u-v)\|_g^2 \,  \textnormal{dvol}_{g}\leq J(u)-J(v)\leq \mu^*\, \vol_g(\{u=0\}\cap B_r^{g}(x_0))\,.
\end{equation}

\noindent\textbf{Step 2:} Using classical elliptic tools (here we do not use the optimality of $u$ anymore), we obtain in this step the reverse estimate\footnote{One can find a different proof in  \cite[Lemma 3.2]{AC81Exi} in the harmonic case; the proof we use here is an adaptation of an argument that has been communicated to us by Antoine Mellet, see \url{https://vimeo.com/118498464}}
\begin{equation}\label{eq:harnack}
\int_{B_r^{g}(x_0)} \|\nabla^{g}(u-v)\|_g^2 \,  \textnormal{dvol}_{g}\geq c\, \vol_g(\{u=0\}\cap B_r^{g}(x_0))\left(\frac{1}{r}\fint_{\partial B_r^{g}(x_0)}u \,  \textnormal{dvol}_{g|_{\partial B_r^{g}(x_0)}}\right)^2 \,,
\end{equation}
for some constant $c>0$.\\
For $r$ small enough  and for any $x_{0}$ in a compact set of $M$ the exponential map $\textnormal{exp}_{x_0}$ is a bi-Lipschitz continuous diffeomorphism from $B^{g}_r(x_0)$ to the ball $B_r$ of radius $r$ in the tangent space, with a Lipschitz constant independent of r and of $x_{0}$. If we denote by $u^*$ and $v^*$ the pull back of the functions $u$ and $v$ in the tangent space, by the Hardy inequality we have 
$$\int_{B_r}\|\nabla(u^*-v^*)\|_e^2\, \textnormal{dvol}_e \geq \frac{1}{4}\int_{B_r}\left(\frac{v^*-u^*}{r-|x|}\right)^2\, \textnormal{dvol}_e\, .$$
Now, using the bi-Lipschitz diffeomorphism given by the exponential map we get the same estimate with a smaller constant $c_{1}>0$: 
$$\int_{B^{g}_r(x_0)} \|\nabla^{g}(u-v)\|_g^2\,  \textnormal{dvol}_{g}\geq c_{1}\int_{B^{g}_r(x_0)}\left(\frac{v-u}{\delta_{r,x_{0}}}\right)^2\,  \textnormal{dvol}_{g}$$
where $\delta_{r,x_{0}}(x)=d_{g}(x,\partial B_r^{g}(x_0))$
 and then
$$\int_{B_r^{g}(x_0)} \|\nabla^{g}(u-v)\|_g^2\,  \textnormal{dvol}_{g} \geq c_{1}\int_{B_r^{g}(x_0)\cap \{u=0\}}\left(\frac{v}{\delta_{r,x_{0}}}\right)^2\,  \textnormal{dvol}_{g}\,.$$
To estimate the last term, it remains to understand the behavior of $v$ in the ball.
We introduce $\phi$ such that $u-\phi\in H^1_{0}(B^{g}_{r}(x_{0}))$ and $\phi$ harmonic in $B^{g}_{r}(x_{0})$, i.e. $\phi$ is the harmonic replacement of $u$ in that geodesic ball. Take a negative constant $k$ such that the Ricci curvature of the manifold $M$ satisfies, in $B^{g}_{r_0}(x_{0})$, $\textnormal{Ric} \ge (n-1) k$.
Then by the first part of Lemma \ref{lem:ricci} we have 
\[
\displaystyle \phi (x_0)\ge \frac{1}{\vol_{g_k}(S^{g_k}_r)} \int_{S^g_r(x_0)} \phi\, \textnormal{dvol}_{g}|_{S^g_r(x_0)}\,.
\]
Similarly to the proof of Corollary \ref{cor:continuity}, we write
\begin{eqnarray*}
 \frac{1}{\vol_{g_k}(S^{g_k}_r)} \int_{S^{g}_r(x_0)} \phi\, \textnormal{dvol}_{g}|_{S^{g}_r(x_0)} & =&  \frac{1}{\vol_g(S^{g}_r(x_0))}\,\frac{\vol_g(S^{g}_r(x_0))}{\vol_{g_k}(S^{g_k}_r)} \int_{S^{g}_r(x_0)} \phi\, \textnormal{dvol}_{g}|_{S^g_r(x_0)}\\
 & \ge & c_{2}\, \frac{1}{\vol_g(S^{g}_r(x_0))} \int_{S^{g}_r(x_0)} \phi\, \textnormal{dvol}_{g}|_{S^{g}_r(x_0)}
 \end{eqnarray*}
for some constant $c_{2}$ depending only on $k$, that is to say uniformly in $r<r_0$ and $x_{0}$. Then 
\[
\phi(x_0) \geq c_{2}\, \fint_{S^g_r(x_0)} \phi\, \textnormal{dvol}_{g}|_{S^g_r(x_0)} 
\]
Using now Harnack's inequality (see for example \cite[Theorem 8.20]{GT01Ell}), there is a constant $c_{3}>0$ independant on $x_{0}$ and $r\leq r_{0}$ such that $\phi\geq c_{3}\phi(x_{0})$ in $B^{g}_{\frac{r}{2}}(x_0)$.
Denoting $c_{4}=c_{2}c_{3}$ we finally get
\begin{equation}\label{eq:est1}
v\geq \phi\geq c_{4}\fint_{\partial B^{g}_{r} (x_0)}u \, \textnormal{dvol}_{g}|_{\partial B^g_r(x_0)}\;\;\;\;\textrm{ in }\;B^{g}_{\frac{r}{2}}(x_0).
\end{equation}
In $B^{g}_{r}(x_0) \setminus B^{g}_{\frac{r}{2}}(x_0)$ we argue as in the proof of the classical Hopf's lemma (see for example \cite[Section 6.4.2]{E10Par}; see also \cite[Lemma 4.1]{W05Opt}). In normal geodesic coordinates we consider the function $w(x)=e^{-\gamma\frac{|x|^2}{r^2}}-e^{-\gamma}$. After computation we have that in the chart representing $B^{g}_{r}(x_0) \setminus B^{g}_{\frac{r}{2}}(x_0)$, the Euclidean Laplacian of $w$ is 
$$\Delta_e w(x)=\frac{2\gamma}{r^2}\left[2\gamma\frac{|x|^2}{r^2}-n\right]e^{-\gamma\frac{|x|^2}{r^2}}\geq\frac{2\gamma}{r^2}\left[\frac{\gamma}{2}-n\right]e^{-\frac{\gamma}{4}}$$
so that, if $\gamma$ is large enough, $-\Delta_e w(x) \leq -k'$, for some positive constant $k'$, independant on $x_{0}$ and $r\leq r_{0}$. Then, choosing $r$ small enough and using \eqref{eq:deltag}, $-\Delta_{g}w\leq 0$ in $B^{g}_{r}(x_0) \setminus B^{g}_{\frac{r}{2}}(x_0)$. On the other hand, we define 
$$\varphi:=\phi-\left[\frac{c_{4}\displaystyle \fint_{S^g_r(x_0)} u\, \textnormal{dvol}_{g}|_{S^g_r(x_0)}}{\displaystyle w\left(\frac{r}{2}\right)}\right] w$$ which is such that $\varphi\geq 0$ on $ \partial B^{g}_{\frac{r}{2}}(x_0)$,  $\varphi=u\geq 0$ on $\partial B^{g}_{r}(x_0)$ and $-\Delta_{g}\varphi\geq 0$. We obtain from maximum principle that 
\begin{equation}\label{eq:est2}
\displaystyle v\geq \phi\geq   \left[ c_{4}\fint_{\partial B^{g}_{r}(x_0)}u  \, \textnormal{dvol}_{g}|_{\partial B^g_r(x_0)}\right]\frac{w}{w\left(\frac{r}{2}\right)}
\end{equation}
in $B^{g}_{r}(x_0) \setminus B^{g}_{\frac{r}{2}}(x_0)$. We finally remark that there is a constant $c_{5}=c_{5}(\gamma)$ such that $\displaystyle{\frac{e^{-\gamma y^2}-e^{-\gamma}}{e^{-\gamma/4}-e^{-\gamma}}}\geq c_{5}(1-y)$ for $y\in [0,1]$, so that $\displaystyle{\frac{w}{w(r/2)}\geq c_{5}\left(1-\frac{|x|}{r}\right)}$. Combining \eqref{eq:est1} and \eqref{eq:est2}
we get that there exists a constant $c_{6}$ independant on $x_{0}$ and $r\leq r_{0}$ such that
\[
v(x)\geq c_{6}\left(\frac{1}{r}\fint_{\partial B_{r}^{g}(x_0)}u \,  \textnormal{dvol}_{g|_{\partial B_r^{g}(x_0)}} \right)(r-d_{g}(x,x_0))= c_{6}\left(\frac{1}{r}\fint_{\partial B_{r}^{g}(x_0)}u \,  \textnormal{dvol}_{g|_{\partial B_r^{g}(x_0)}} \right) \delta_{r_{0},r_{0}}(x)
\]
which leads to the estimate \eqref{eq:harnack}.

\noindent{\bf Conclusion:} Combining \eqref{eq:estimate1} and \eqref{eq:harnack}, we obtain that if $\displaystyle{\frac{1}{r}\fint_{\partial B_{r}^{g}(x_0)}u\,  \textnormal{dvol}_{g|_{\partial B_r^{g}(x_0)}} > \sqrt{\frac{\mu^*}{c}}}$ then $u=v$  almost everywhere on $B_{r}^{g}(x_0)$. Using the representation \eqref{eq:pointwise} for both $u$ and $v$, we deduce that $u=v>0$ everywhere on $B_{r}^{g}(x_0)$, which concludes the proof.
\qed
\end{proof}

\begin{remark}\label{rk:sign}
Note that from \eqref{eq:estimate1} and using the representation \eqref{eq:pointwise}, we deduce that if $u>0$ almost everywhere in $B_{r}^g(x_0)$, then $u=v$ everywhere in $B_{r}^g(x_0)$, and therefore $u>0$ everywhere in $B_{r}^g(x_0)$. This fact will be used in Sections \ref{ssect:pen} and \ref{ssect:positivity}. 
\end{remark}

\begin{corollary}\label{cor:lip}
Let $M$ be a compact Riemannian manifold, $m\in(0,\vol_g(M))$ and $u$ a solution of \eqref{eq:pbfonc}. Then $u$ is Lipschitz continuous in $M$.
\end{corollary}

\begin{remark}\label{rk:noncompact}
If $M$ is noncompact, one deduces that $u$ is locally Lipschitz in $M$.
Also, combined with Proposition \ref{prop:saturation}, one deduces that If $M$ is connected and $\Om^*$ solves \eqref{eq:pbforme2}, then $\Om^*$ is an open set and hence solves \eqref{eq:pbforme}. See Remark \ref{rk:connectedness} for a counterexample in the disconnected case. 
\end{remark}

\begin{remark}
In order to prove that $\Om_{u}$ is open (where $u$ solves \eqref{eq:pbfonc}), we only need to prove that the eigenfunction $u$ is continuous on $M$ (rather than Lipschitz continuous), which can be obtained in several ways, see for example the elementary proof in \cite[Lemma 3.8]{BHP05Lip}, or also \cite[Remark 3.10]{BHP05Lip} and \cite[Section 3]{W05Opt} based on a classical method from Morrey \cite{M08Mul}; this last method gives H\"older regularity for any order $\alpha<1$. Nevertheless, we will need Lipschitz continuity in the following sections.
\end{remark}

\noindent\textbf{Proof of Corollary \ref{cor:lip}:} 
The fact that $\Om_{u}$ is open is a consequence of Corollary \ref{cor:continuity} and Lemma \ref{lem:lip}. Indeed, let $x_{0}\in M$ such that $u(x_{0})>0$: then from \eqref{eq:meanvalue} we get that ($c'$ and $r_{0}$ coming from Corollary \ref{cor:continuity}) $$\fint_{\partial B^g_{r}(x_{0})}u \,  \textnormal{dvol}_{g|_{\partial B_r^{g}(x_0)}}+r^2 \geq \frac{u(x_{0})}{c'}$$ for every $r\leq r_{0}$. So $$\frac{1}{r}\fint_{\partial B^{g}_{r}(x_0)}u\,  \textnormal{dvol}_{g|_{\partial B_r^{g}(x_0)}} \geq \frac{u(x_{0})}{c'r}-r$$
and
with Lemma \ref{lem:lip}, we obtain $u>0$ in $B^{g}_{r}(x_{0})$ for $r$ small enough.

We are now in position to prove the Lipschitz continuity of $u$. Since $\Om_{u}$ is an open set, we can choose the maximal radius $r$ such that the geodesic ball $B^{g}_{r}(x_{0})$ is included in $\Om_{u}$. As $u$ is smooth inside $\Om_{u}$, we may assume $r<r_{0}$. Thanks to Lemma \ref{lem:lip},
$$\frac{1}{r+\delta}\fint_{\partial B^{g}_{r+\delta}(x_{0})}u\,  \textnormal{dvol}_{g|_{\partial B_{r+\delta}^{g}(x_0)}}  \leq C\,,$$ where $C$ is introduced in Lemma \ref{lem:lip}, and for all small $\delta>0$. Therefore $$\frac{1}{r}\fint_{\partial B^{g}_r(x_{0})}u\,  \textnormal{dvol}_{g|_{\partial B_r^{g}(x_0)}}  \leq C.$$
From this estimate, we deduce now that the gradient of $u$ is bounded:
we introduce $w$ defined by
$$\left\{\begin{array}{cccl}
-\Delta_{g} w & = &\lambda(m) u &\textrm{ in } B^{g}_r(x_{0}),\\
w & = & 0 &\textrm{ outside } B^{g}_r(x_{0}).
\end{array}
\right.
$$
On one hand, as $u-w$ is harmonic, introducing $G$ the Green function with Dirichlet boundary condition on $\partial B^{g}_{r}(x_{0})$ (see \cite[Theorem 4.17]{A98Som}), we have
\begin{equation}\label{eq:green}
\begin{split}
\;\;\|\nabla^{g}(u-w)(x)\|_g\leq \int_{\partial B^{g}_{r}(x_0)} \underbrace{\|(\nabla^{g}_{y})^2G(x,y)\|_g}_{\leq kd(x,y)^{-n}}u(y)\,  \textnormal{dvol}_{g|_{\partial B_r^{g}(x_0)}}(y) \leq\\
\leq k\,\left(\frac{r}{2}\right)^{-n}\int_{\partial B_{r}(x_{0})}u(y)\,  \textnormal{dvol}_{g|_{\partial B_r^{g}(x_0)}}(y) \leq kC\,,
\end{split}
\end{equation}
for $x \in B^{g}_{r}(x_{0})$, where $k$ is a constant depending on the distance of $x$ to $\partial B^{g}_{r}(x_{0})$. This constant $k$ is not depending on $x$, if we take $x \in B^{g}_{\frac{r}{2}}(x_{0})$ (according to \cite[4.10]{A98Som}, $k$ is not depending on $x$ if the injectivity radius $i_x$ at $x$ with respect to the manifold $B^{g}_{r}(x_{0})$ is bigger than a positive constant $\rho$, and this is the case for $x \in B^{g}_{\frac{r}{2}}(x_{0})$).\\
We use now the scaling argument introduced at the beginning of Section \ref{sect:reg} to prove:
\begin{equation}\label{eq:perturb}
\|\nabla^{g} w(x)\|_g\leq Cr\lambda(m)\|u\|_{L^\infty(B_{r}(x_{0}))}
\end{equation}
for $x \in B_{\frac{r}{2}}^g(x_0)$. 
Indeed, considering $y \in B_{\frac12}$, the Euclidean ball of radius ${\frac12}$, and endowing on $B_{\frac12}$ the metric $\bar g$, let us define $\widetilde{w}(y)=w(\textnormal{exp}^g_x(r\Theta(y)))$, with $\Theta$ defined in \eqref{eq:theta}. We have
\[
{-\Delta_{\bar g} \widetilde w (y) = -r^2\, \Delta_{g} w (\textnormal{exp}^g_x(ry))} = r^2\lambda(m)u(\textnormal{exp}^g_x(r\Theta(y)))
\]
where $\bar g = r^{-2} g$, so by interior gradient estimates, there is a constant $C$ such that
$$|\nabla^e \widetilde w(0) | \leq Cr^2\lambda(m)\|u\|_{L^\infty(B_{r}(x_{0}))}$$
and then also
$$\|\nabla^{\bar g} \widetilde w(0) \|_{\bar g} \leq Cr^2\lambda(m)\|u\|_{L^\infty(B_{r}(x_{0}))}\,.$$
Now, using that $\nabla^{g} w(x)=\frac{1}{r}	\nabla^{\bar g} \widetilde{w}(0)$, we obtain \eqref{eq:perturb}. Combining with \eqref{eq:green}, we obtain a uniform bound on $\nabla^{g} u(x)$ for $x \in B^{g}_{\frac{r}{2}}(x_{0})$, and therefore $u$ is Lipschitz continuous.

\qed

\subsection{Refined penalization of the volume constraint}\label{ssect:pen}

In order to investigate further regularity properties of the free boundary, we need to prove a more involved version (though localized) of the penalization property stated in Proposition \ref{prop:pen1}:
let $u$ be a solution of \eqref{eq:pbfonc}, and $B^{g}_{r_{0}}(x_{0})$ be a geodesic ball
 centered at $x_{0}\in\partial\Omega_u$. We define
\begin{equation}\label{eq:F}
\F=\F(u,x_{0},{r_0})=\{v\in H^1(M), u-v\in H^1_0(B^{g}_{r_0}(x_{0}))\}.
\end{equation}
For $h>0$, we denote by $\mu_-(x_{0},{r_0},h)$ the largest $\mu_-\geq 0$ such that,
\begin{equation}\label{eq:mu-}
\forall\; v \in\F\textrm{ such that }  m-h\le \vol_g(\Omega_v)\le  m,\;
J(u)+\mu_-\vol_g(\Omega_u)\le J(v)+\mu_-\vol_g(\Omega_v).
\end{equation}
We also define $\mu_+(x_{0},{r_0},h)$ as the smallest $\mu_+\geq 0$ such that,
\begin{equation}\label{eq:mu+}
\forall\; v \in\F\textrm{ such that }  m\le \vol_g(\Omega_v)\le  m+h,\;
J(u)+\mu_+\vol_g(\Omega_u)\le J(v)+\mu_+\vol_g(\Omega_v).
\end{equation}
The following result deals with the aymptotic behaviour of these penalization coefficients:
\begin{prop}\label{prop:pen}
Let $u$  be a solution of \eqref{eq:pbfonc}. There exists $\Lambda=\Lambda_{u}\geq 0$  such that, for every $x_{0}\in \partial\Om_{u}$ and ${r_0}$ small enough, there exists $h_0>0$ such that,
\begin{multline}\label{limmu}
\forall\; h\in(0,h_0),\; \mu_-(x_{0},{r_0},h)\le\Lambda\le\mu_+(x_{0},{r_0},h)<+\infty, 
\\\textrm{ and, moreover, }
\lim_{h\to 0}\mu_+(x_{0},{r_0},h)=\lim_{h\to 0}\mu_-(x_{0},{r_0},h)=\Lambda.
\end{multline}
\end{prop}

We start writing a weak optimality condition
for the constrained problem \eqref{eq:pbfonc}; in this way, we can define $\Lambda$ as a Lagrange multiplier.
\begin{lemma}[Euler-Lagrange equation]\label{lem:euler}
If $u$ is a solution of \eqref{eq:pbfonc}, then there exists
$\Lambda=\Lambda_{u}\geq 0$ such that, 
\begin{equation}\label{eq:EL}\forall \Phi\in C^{\infty}(M,TM),\;\;
\int_M \Big[2\, g(D\Phi\nabla^{g} u, \nabla^{g} u)+(\lambda(m) 
 u^2-  \|\nabla^{g} u\|_g^2) \, \div_{g} \Phi\Big] \,  \textnormal{dvol}_{g}= \Lambda\int_{\Om_u}\div_{g} \Phi\ \,  \textnormal{dvol}_{g}\,. 
\end{equation}
\end{lemma}
\begin{proof} {For $\Phi\in C^{\infty}(M,TM)$ and $t\in\R$,
we consider $u_{t}(x)=u(\textnormal{exp}_x(t\Phi(x))\in H^1(M)$.
If $t$ is small enough, $x\mapsto \textnormal{exp}_x(t\Phi(x))$ is a $C^1$-diffeomorphism of $M$}, so with a change of variable we get:
\[\vol_g(\Om_{u_{t}})=\vol_g(\Om_u)-t\int_{\Om_u}\div_{g} \Phi \,  \textnormal{dvol}_{g} + o(t), \]
$$J(u_{t})=J(u)+t
\int_M \Big[2g(D\Phi\nabla^{g} u,\nabla^{g} u)-  \|\nabla^{g} u\|_g^2\div_{g}\Phi+\lambda(m)
 u^2 \div_{g} \Phi \Big] \,  \textnormal{dvol}_{g} + o(t).$$
 Moreover, the linear form $\Phi\mapsto \int_{\Om_{u}} \div_{g}\Phi  \,  \textnormal{dvol}_{g}$ does not vanish
 , so we can apply the Karush-Kuhn-Tucker condition for the minimization of $J$ among $v\in H^1(M)$ with the constraint $\vol_g(\Om_{v})\leq  m$: we get the existence of $\Lambda\in\R_{+}$ such that
$$\frac{d}{dt}J(u_{t})_{|t=0}=-\Lambda\frac{d}{dt}\vol_g(\Om_{u_{t}})_{|t=0},\;\;\;\forall \Phi\in C^\infty(M,TM),$$
which concludes the proof.
\qed
\end{proof}

\begin{remark}
We can rewrite this Euler-Lagrange equation in the following way:
$$\forall \Phi\in C^{\infty}(M,TM),\;\;\lim_{\eps\to0} \int_{\partial\{u>\eps\}}g((\|\nabla^{g} u\|_g^2-\Lambda) \Phi, \nu_{\eps}) \,  \textnormal{dvol}_{g}=0,$$
where $\nu_{\eps}$ denotes the outward unit normal to $\partial\{u>\eps\}$, so this property can be seen as a very weak formulation of the extremality condition for $\lambda_{1}$, without regularity assumption. 
Indeed, 
 we first notice that inside $\Om_{u}$,
$$2 g(D\Phi\nabla^{g} u,\nabla^{g} u)-g(\nabla^{g}(\lambda(m)u^2-\|\nabla^{g} u\|_g^2),\Phi)=2\div_{g}[g(\Phi,\nabla^{g} u)\nabla^{g} u]-2g((\underbrace{\Delta_{g} u+\lambda(m)u}_{=0})\Phi,\nabla^{g} u),$$
so using Gauss-Green formula, we obtain
\begin{multline*}
\int_M \Big[2 g(D\Phi\nabla^{g} u,\nabla^{g} u)+(\lambda(m) 
 u^2-  \|\nabla^{g} u\|_g^2) \div_{g}\Phi\Big] \,  \textnormal{dvol}_{g}\\=\lim_{\eps\to0} \int_{\partial\{u>\eps\}}\ g(2g(\Phi,\nabla^{g} u)\nabla^{g} u+(\lambda(m)u^2-\|\nabla^{g} u\|_g^2)\Phi  , \nu_{\eps}) \,  \textnormal{dvol}_{g}
 =\lim_{\eps\to0} \int_{\partial\{u>\eps\}}\|\nabla^{g} u\|_g^2 g(\Phi, \nu_{\eps}) \,  \textnormal{dvol}_{g} 
 \end{multline*}
since $\nabla^{g} u=\|\nabla^{g} u\|_g\nu_{\eps}$ on $\{u=\eps\}$.
\end{remark}

We are now in position to prove Proposition \ref{prop:pen}: the proof is very similar to \cite[Theorem 1.5]{BL09Reg}, and the adaptation to the framework of a manifold requires very little change. Therefore, we only sketch the argument; see also \cite[Section 3.2.3]{H17Sha} for a heuristical description of the argument and \cite{RTV18Exi} in a different framework. Until the end of this section, $u$ denotes a solution of \eqref{eq:pbfonc}, $x_{0}\in\partial\Om_u
$, ${r_0}>0$, and $\F$ is defined in \eqref{eq:F}; we denote $\mu_{\pm}(h)$ instead of $\mu_{\pm}(x_{0},{r_0},h)$ to simplify the presentation and as no confusion is possible, and we denote $B^{g}_{r}$ for $B^{g}_{r}(x_{0})$. 
We first need the following lemma which helps obtaining existence results:
\begin{lemma}
There exists a constant $C=C(u)$ such that
for ${r_0}$ small enough,
\[\forall v\in\F,\;\; J(v)\geq \frac{1}{2}\int_{B^{g}_{r_0}}\|\nabla^{g} v\|_g^2 \,  \textnormal{dvol}_{g}  -C. \]
\end{lemma}
\begin{proof}
Let $v\in\F$; since $u-v\in H^1_0(B^{g}_{r_0})$
\begin{eqnarray*}
J(v) &\geq& \int_M \|\nabla^{g} v\|_g^2 \,  \textnormal{dvol}_{g} -
\lambda(m)\|v\|_{L^2(M)}^2\geq \int_M \|\nabla^{g} v\|_g^2 \,  \textnormal{dvol}_{g} -
2\lambda(m)\left(\|u-v\|_{L^2(B^{g}_{{r_0}})}^2+\|u\|_{L^2(M)}^2\right)\\
&\geq&\int_M \|\nabla^{g} v\|_g^2 \,  \textnormal{dvol}_{g} -
2\lambda(m)\left(\frac{\|\nabla^{g}(u-v)\|_{L^2(B^{g}_{{r_0}})}^2}{\lambda_{1}(B^{g}_{{r_0}})}+\frac{\|\nabla^{g} u\|_{L^2(M)}^2}{\lambda_{ m}}\right)\\
&\geq&\int_M \|\nabla^{g} v\|_g^2 \,  \textnormal{dvol}_{g} -
4\lambda(m)\left(\frac{\|\nabla^{g} u\|_{L^2(B^{g}_{{r_0}})}^2+\|\nabla^{g} v\|_{L^2(B^{g}_{{r_0}})}^2}{\lambda_{1}(B^{g}_{{r_0}})}\right)-2\|\nabla^{g} u\|_{L^2(M)}^2\\
&\geq&\left(\int_{B^{g}_{{r_0}}} \|\nabla^{g} v\|_g^2 \,  \textnormal{dvol}_{g}\right)\left(1-4\frac{\lambda_{ m}}{\lambda_1(B^{g}_{r_0})}\right) -2\|\nabla^{g} u\|^2_{L^2(M)}\left(1+2\frac{\lambda_{ m}}{\lambda_{1}(B^{g}_{{r_0}})}\right)
\end{eqnarray*}
which gives the result if $\lambda_{1}(B^{g}_{{r_0}})$ is large enough, which is true if ${r_0}$ is small enough.
\qed
\end{proof}
\begin{remark}\label{jbound} This lemma implies that $J$ is bounded from below on $\F$, and moreover that if $v_n\in\F$ is a sequence such that $J(v_n)$ is bounded,
then $\|\nabla^{g} v_n\|_{L^2(B^{g}_{r_0})}$ is also bounded.
Since $v_n=u$ outside $B^{g}_{r_0}$ we deduce that $v_n$ weakly converges up to a sub-sequence.
\end{remark}

\noindent\textbf{Sketch of proof of Proposition \ref{prop:pen}: }Let $\Lambda\geq0$ be as in
Lemma \ref{lem:euler}. The proof is divided in three steps. \\[-2mm]

\noindent\textbf{- First step: $\Lambda\le\mu_+(h)<+\infty$.}
To prove that $\mu_+(h)$ is finite, we first notice that  $0<\vol_g(\Om_u\cap B^{g}_{r_0})<\vol_g(B^{g}_{r_0})$, see \cite[Lemma 2.5]{BL09Reg} and Remark \ref{rk:sign}. 
We consider then $h\in(0,\vol_g(B^{g}_{r_0})-\vol_g(\Om_u\cap B^{g}_{r_0}))$ (and so, if
$v\in\F$ with $\vol_g(\Om_v)\le  m+h$, then $\vol_g(\Om_v\cap B^{g}_{r_0})<\vol_g(B^{g}_{r_0})$), and the optimization problem
\begin{equation*}
\min\left\{J(v)+\mu \vol_g(\Om_v), v\in\F,  m\leq \vol_g(\Om_{v})\le  m+h \right\}.
\end{equation*}
Using Remark \ref{jbound}, we have existence of a solution $v_{\mu}$. 
If $\vol_g(\Om_{v_{\mu_{0}}})=  m$ for some $\mu_{0}$ then $u$ is a solution to $(\ref{eq:mu+})$ 
with $\mu_{0}$ and therefore $\mu_{+}(h)$ is finite: we will suppose to the contrary that $\vol_g(\Om_{v_\mu})>m$ for all $\mu$.
As $v_{\mu}$ is solution to 
\[J(v_\mu) =\min \Big\{J(w)\;,\;w\in\F,  m\leq\vol_g(\Om_w)\leq \vol_g(\Om_{v_{\mu}})\Big\}, \]
we can write an Euler-Lagrange equation for $v_\mu$ with a similar proof to Lemma \ref{lem:euler} and there exists
$\Lambda_\mu\geq 0$ such that \eqref{eq:EL} is true for $v_{\mu}$ in place of $u$ and for $\Phi\in C^{\infty}_c(B^{g}_{r_0},TM)$.
 On one hand, using the optimality condition for $v_{\mu}$ we obtain that $\Lambda_\mu\geq \mu$, while on the other hand, using compactness and extracting from \eqref{eq:EL} formulas for $\Lambda$ and $\Lambda_{\mu}$ we get 
%
$\lim_{\mu_{n}\to\infty}\Lambda_{\mu_{n}}=\Lambda$ for a sequence $\mu_{n}$ going to $+\infty$. This leads to a contradiction, so  $\mu_{+}(h)$ is finite. 
%
%
To conclude this first step, we show that $\Lambda\le\mu_+(h)$, again using the optimality condition for $u$.\\[-2mm]

\noindent\textbf{- Second step: $\lim_{h\to 0}\mu_+(h)=\Lambda$. }
 We first see that $\mu_+(h)>0$ for $h>0$. Indeed, if $\mu_+(h)=0$ we can use 
\[\textrm{for every }\varphi\in C^{\infty}_0(B^{g}_{r_0})\textrm{ with }\vol_g(\{\varphi\neq 0\})<h,\; J(u)\le J(u+t\varphi), \]
which leads to
$-\Delta_{g} u=\lambda(m) u \textrm{ in } B^{g}_{r_0},$
and contradicts $\vol_g(\Om_u\cap B^{g}_{r_0})<\vol_g(B^{g}_{r_0}).$
Let $\eps>0$. Since $h\mapsto\mu_+(h)$ is non-decreasing, we just have to
see that $\mu_+(h)\le\Lambda+\eps$ for some $h>0$.
If $\Lambda>0$, we restrict to $\eps<\Lambda$ and define
$\mu_{\eps}(h):=\mu_+(h)-\eps>0$; if $\Lambda=0$, we define
$\mu_{\eps}(h)=\mu_+(h)/2>0$.
We apply the same strategy as in the first step to $v_h$ solution to
\begin{equation*}
\min\left\{J(v)+\mu_{\eps}(h)\vol_g(\Om_{v}), \;\;v\in\F, \vol_g(\Om_v)\le  m+h\right\}.
\end{equation*}
We notive first that by definition of $\mu_{+}(h)$ we have $\vol_g(\Om_{v_{h}})>m$. Denoting then $\Lambda_{h}$ the Lagrange multiplier associated to $v_{h}$, we prove as before that $\Lambda_{h}\geq \mu_{\eps}(h)$ and also that $\lim\Lambda_{h_{n}}=\Lambda$ for some sequence $h_{n}$ going to 0. This leads to $\mu_+(h)\leq\Lambda+2\eps$ for some $h$ and concludes this step.\\[-2mm]

\noindent\textbf{- Third step: $\lim_{h\to 0}\mu_-(h)=\Lambda$. }
As in the first step, we first see that $\mu_-(h)\le\Lambda$. 
As in the previous step, we study $v_h$ solutions of the following minimization problem
\begin{equation*}\label{vnmu}
\min \left\{J(w)+(\mu_-(h)+\eps)\vol_g(\Om_{w}), \;\;w\in\F,\  m-h\leq \vol_g(\Om_w)\le  m\right\}.
\end{equation*}
to deduce that $\lim_{h\to 0}\mu_-(h)=\Lambda$ and this concludes the proof. \qed
%

\subsection{Positivity of the Lagrange multiplier}\label{ssect:positivity}

\begin{prop}\label{prop:positivity}
Let $u$ a solution of \eqref{eq:pbfonc}, and $\Lambda_{u}$ given in Proposition \ref{prop:pen}. Then $\Lambda_{u}>0$.
\end{prop}

We follow \cite[Proof of Proposition
6.1]{B04Reg}, though we slightly simplify the presentation. For another argument, see \cite[Appendix A]{RTV18Exi} which relies only on the use of \eqref{eq:EL}.

\medskip

\begin{proof}
We argue by contradiction and suppose that $\Lambda=0$.  Our aim is to prove that under such assumption
$$-\Delta_{g} u =\lambda(m)u\;\;\;\textrm{ in the sense of distribution in }M$$
which asserts that the measure $\Delta_{g} u$ does not charge $\partial\Om_{u}$.  We introduce $(\omega_{n})_{n\in\N}$ an increasing sequence of smooth open sets such that
$$\left\{x\in\Om_{u}, d_{g}(x,\partial\Om_{u})\geq \frac{1}{n}\right\}\subset \omega_{n}\subset \overline{\omega_{n}}\subset\Om_{u},$$
so that $\bigcup_{n}\omega_{n}=\Om_{u}$.
Take $B= B^{g}_{r_0}(x)$ for some $x\in \partial \Omega_u$, and $r_0>0$.

\noindent{\bf Step 1: the gradient of $u$ vanishes near $\partial\Om_{u}$.} Let $x_{0}\in(\Om_{u}\cap B)\setminus \omega_{n}$. Denote $r$ the largest number such that $B^{g}_{r}(x_{0})\subset \Om_{u}$. We use the function $v$ defined in \eqref{eq:v1}; similarly to the proof in Lemma \ref{lem:lip}, we obtain
 from the definition of $\mu_{+}$:
\begin{equation}\label{eq:grad}
\frac{1}{\rho}\fint_{\partial B^{g}_\rho(x_{0})}u \,  \textnormal{dvol}_{g|_{\partial B_\rho^{g}(x_0)}}  \le C\sqrt{\mu_{+}(\vol_g(B^{g}_{\rho}(x_{0})))} \;\;\;\;\forall B^{g}_{\rho}(x_{0})\subset B\textrm{ such that }\vol_g(\{u=0\}\cap B^{g}_\rho(x_{0}))> 0
\end{equation}
and from the fact that $\Lambda=0$ and using Proposition \ref{prop:pen}, the right hand side converges to $0$ when $\rho\to 0$.
By definition of $r$, for every $\delta>0$, $\{u=0\}\cap B^{g}_{r+\delta}(x_{0})\neq \emptyset$, so as in Remark \ref{rk:sign} this implies that $\vol_g(\{u=0\}\cap B^{g}_{r+\delta}(x_{0}))>0$
, and therefore we can use \eqref{eq:grad} and the monotonicity of $\mu_{+}$:
$$\frac{1}{r+\delta}\fint_{\partial B^{g}_{r+\delta}(x_{0})}u \,  \textnormal{dvol}_{g|_{\partial B_{r+\delta}^{g}(x_0)}} \le  C\sqrt{\mu_{+}(\vol_g(B^{g}_{r+\delta}(x_{0})))}\leq C\sqrt{\mu_{+}(\vol_g(B^{g}_{2r}(x_{0})))} \;\;\;\;\forall \delta<r.$$
We can let $\delta$ go to 0. Then, as in the proof of Corollary \ref{cor:lip}, we obtain
$$\|\nabla^{g} u(x_{0})\|_g\leq C\left(\frac{1}{r}\fint_{\partial B^{g}_{r}(x_0)}u \,  \textnormal{dvol}_{g|_{\partial B_{r}^{g}(x_0)}} +r\right)\leq  C\left(\sqrt{\mu_{+}(\vol_g(B^{g}_{2/n}(x_{0})))}+\frac{1}{n}\right)$$
and therefore $\|\nabla^{g} u\|_{\infty,B\setminus \omega_{n}}$ converges to 0 when $n$ goes to $\infty$.

\noindent{\bf Step 2.} Thanks to Proposition \ref{prop:reg0}, we know that $\Delta_{g} u$ is a Radon measure. Let $\varphi\in C^\infty_{c}(B)$. We can write
\begin{equation}\label{eq:split}
\int_B\varphi \, \textnormal{d}(\Delta_{g} u) =\int_{B\cap \omega_{n}}\varphi \, \textnormal{d}(\Delta_{g} u)  +\int_{B\setminus\omega_{n}}\varphi \, \textnormal{d}(\Delta_{g} u)\end{equation}
The first term is equal to $\lambda(m)\int_{B\cap \omega_{n}}u\varphi$ and converges to $\lambda(m)\int_{B}u\varphi$
when $n$ goes to $\infty$. We are therefore aiming at proving that the second term in \eqref{eq:split} converges to 0. We introduce the vector valued function $H_{n}$ defined by
$$
H_{n}\in H^1(M,\R^n), \;\;\;\;\;\;
\left\{
\begin{array}{cccl}
\Delta_{g} H_{n}&=&0&\textrm{ in }\omega_{n}\\
H_{n}&=&\nabla^{g} u&\textrm{ in }\omega_{n}^c
\end{array}
\right.
$$
Then
\begin{equation}\label{eq:split2}
\int_{B\setminus \omega_{n}}\varphi \, \textnormal{d}(\Delta_{g} u)=\int_{B\setminus \omega_{n}}\varphi \, \textnormal{d}(\div_{g} H_{n}) =\int_{B}\varphi \, \textnormal{d}(\div_{g} H_{n})-\int_{B\cap \omega_{n}}\varphi \, \textnormal{d}(\div_{g} H_{n}) \end{equation}
The first term is equal to $-\int_{B}H_{n}\varphi  \,  \textnormal{dvol}_{g}$. 
 As $H_{n}$ uniformly converges to 0 on $B\cap \omega_{n}^c$ and is harmonic on $\omega_{n}$, by maximum principle $H_{n}$ converges to 0 uniformly on any compact subset of $B$ and therefore this first term converges to 0.
In order to deal with the second term, we couple the two following statements:
\begin{itemize}
\item first, we know that $\div_{g}(H_{n})$ is harmonic in $\omega_{n}$, and its trace on $\partial\omega _{n}$ is $\Delta_{g} u=-\lambda(m)u$, so
it converges uniformly to 0 on any compact set of $\Om_{u}\cap B$;
\item second, as we have
$$\left\{
\begin{array}{cccl}
\Delta_{g} (H_{n}-\nabla^{g} u)&=&\lambda(m)u&\textrm{ in }\omega_{n}\\
H_{n}-\nabla^{g} u&=&0&\textrm{ on }\partial \omega_{n}
\end{array}
\right.
$$
we get {the classical bound $\|\nabla^{g}(H_{n}-\nabla^{g} u)\|_{L^2(B\cap \omega_{n})}\leq C\lambda(m)\|u\|_{L^2(\omega_{n}\cap B)}\leq \lambda(m)$}, 
for some constant $C$, which leads to
$$\|\div_{g}(H_{n})\|_{L^2(\omega_{n}\cap B)}\leq \|\div_{g}(H_{n}-\nabla^{g} u)\|_{L^2(\omega_{n}\cap B)}+\|\Delta_{g} u\|_{L^2(\omega_{n}\cap B)}\leq (1+C)\lambda(m).$$
\end{itemize}
Combing these two statements, we deduce with the Lebesgue convergence Theorem that the last term in \eqref{eq:split2} converges to 0, which concludes the proof.
\qed
\end{proof}

\subsection{Non-degeneracy of the state function}

When we proved the Lipschitz-continuity of $u$, we obtained an estimate from above of the gradient of $u$ near the boundary of $\Om_{u}$. We are going to prove a similar estimate from below, which says that in a weak sense the gradient of $u$ cannot vanish near $\partial\Om_{u}$. To that end we use the penalization from below proven in Sections \ref{ssect:pen} and \ref{ssect:positivity}. We use a strategy from \cite[Lemma 3.4]{AC81Exi}.

\medskip

Let $u, x_{0}$ and $r_0$ as in Section \ref{ssect:pen}.

\begin{lemma}\label{lem:nondeg}
There exist $c$ such that, for every ball $B^{g}_r(x)\subset B^{g}_{r_{0}}(x_0)$ with $r$ small enough,
\begin{equation}
\;\|u\|_{\infty,B^{g}_r(x)} \le cr
\;\;\;\;\Longrightarrow\;\;\; u \equiv 0 \textrm{ on } B^{g}_{r/2}(x)\,.
\end{equation}
\end{lemma}

The strategy here is the opposite from Section \ref{ssect:lip}, in the sense that we will consider a test function who is vanishing in a small ball. We will therefore use the penalization from below that has been proven in Sections \ref{ssect:pen} and \ref{ssect:positivity}.
Compare to \cite{AC81Exi} and \cite{BL09Reg}, the statement deals with the $L^\infty$-norm instead of the $L^1$-average as in Section \ref{ssect:lip}. This looks like a weaker result, but before proving Lemma \ref{lem:nondeg}, we show in the next statement that this other version is a classical corollary. This was noticed already in \cite{ACF84Afr} where they state a result with $L^p$-average; even though it is likely possible to directly obtain the following statement in our framework, we felt that the construction of the test function to obtain Lemma \ref{lem:nondeg} is slightly easier, see also \cite{RTV18Exi}.

\begin{corollary}\label{cor:nondeg}
There exists $c$ such that, for every geodesic ball $B^{g}_r(x)\subset B^{g}_{r_{0}}(x_0)$ with $r$ small enough,
\begin{equation}
\;\frac{1}{r}\fint_{\partial B^{g}_r(x)} u  \,  \textnormal{dvol}_{g|_{\partial B_{r}^{g}(x)}} \le c
\;\;\;\;\Longrightarrow\;\;\; u \equiv 0 \textrm{ on } B^{g}_{r/4}(x)
\end{equation}
\end{corollary}

\begin{proof}
Let $B^{g}_r(x)\subset B^{g}_{r_{0}}(x_0)$. 
As in the proof of Corollary \ref{cor:continuity}, we introduce $w(y)=|y|^2$ in local coordinates centered at $x$, and $\widetilde{u}=u+\lambda(m)\|u\|_{\infty}w$ is subharmonic with respect to the metric $g$. Denoting $\phi$ the harmonic replacement of $\widetilde{u}$ with respect to the metric $g$, we have:
\begin{itemize}
\item from maximum principle, $0\leq u\leq \widetilde{u}\leq \phi$,
\item from Lemma \ref{lem:ricci} (choosing $k$ such that the condition on the curvature is satisfied in $B^{g}_{r_{0}}(x_0)$, and $r_0$ small enough such that the area of $S_{r}^{g_{k}}(x)$ is close enough to the area of $S_{r}^{g}(x)$ for $r<r_0$), there exists a constant $c_{1}$ independant of $x,r$ such that $\phi(x)\leq c_{1}\fint_{\partial B^{g}_r(x)} \widetilde{u}  \,  \textnormal{dvol}_{g|_{\partial B_{r}^{g}(x)}}$
\item from Harnack inequality (\cite[Theorem 8.20]{GT01Ell}, there exists $c_{2}$ also independant of $x,r$ such that $\phi\leq c_{2}\phi(x)$ in $B_{r/2}(x)$.
\end{itemize}
Therefore there exists $c'$ such that 
$$\|u\|_{\infty,B^{g}_{r/2}(x)}\leq c'\left[\fint_{\partial B^{g}_r(x)} u  \,  \textnormal{dvol}_{g|_{\partial B_{r}^{g}(x)}}+r^2\right].$$
We can therefore apply Lemma \ref{lem:nondeg} in $B^{g}_{r/2}(x)$ and obtain that $u$ vanishes on $B^{g}_{r/4}(x)$.
%
\qed \end{proof}

\noindent{\bf Proof of Lemma \ref{lem:nondeg}.}
We introduce $w\in H^1(B^{g}_r(x))$ such that:
$$\left\{\begin{array}{rcll}
-\Delta_{g} w & = & \beta &\textrm { in } B^{g}_r(x)\setminus\overline{B^{g}_{r/2}(x)}\\
w & = &\|u\|_{\infty,B^{g}_r(x)}  &\textrm{ in } \partial B^{g}_r(x) \\
w & = & 0 &\textrm{ in } B^{g}_{r/2}(x).
\end{array}
\right.$$
where $\beta\in (0,\infty)$ will be chosen later. Then we define $$v=\left\{\begin{array}{l}\min\{u,w\} \textrm{ in }B^{g}_r(x)\\u\textrm{ in }M\setminus B^{g}_r(x)\end{array}\right.\, .$$ As $u\leq w$ on $\partial B^{g}_r(x)$, $v$ has no discontinuity across $\partial B^{g}_r(x)$ and therefore $v\in H^1(M)$ can be used as a test function. Moreover, by construction 
$\Om_{v}=\Om_{u}\setminus B^{g}_{r/2}(x)$, so we can use the lower penalization given by Proposition \ref{prop:pen} with  $h=\vol_g(B^{g}_{r/2}(x))$ and obtain, since $u=v$ outside $B^{g}_r(x)$,
$$
\int_{B^{g}_r(x)}\|\nabla^{g} u\|_g^2  \,  \textnormal{dvol}_{g} -\lambda(m)\int_{B^{g}_r(x)}u^2 \,  \textnormal{dvol}_{g} + \mu_-(h) \vol_g(\Om_u\cap B^{g}_r(x))
\le$$ $$\le \int_{B^{g}_r(x)}\|\nabla^{g} v\|_g^2 \,  \textnormal{dvol}_{g}-\lambda(m)\int_{B^{g}_r(x)}v^2 \,  \textnormal{dvol}_{g}
+ \mu_-(h) \vol_g(\Om_v\cap B^{g}_r(x)),
$$
which can be written:
\begin{equation}
 \int_{B^{g}_{r/2}(x)}\|\nabla^{g} u\|_g^2 \,  \textnormal{dvol}_{g}+\mu_-(h) \vol_g(\Om_u\cap B^{g}_{r/2}(x))
\le $$ $$\le \int_{B^{g}_r(x)\setminus B^{g}_{r/2}(x)} \left(\|\nabla^{g} v\|_g^2-\|\nabla^{g} u\|_g^2\right)
+\lambda(m)\int_{B^{g}_r(x)\setminus B^{g}_{r/2}(x)}(u^2-v^2) \,  \textnormal{dvol}_{g}
+\lambda(m)\int_{B^{g}_{r/2}(x)} u^2 \,  \textnormal{dvol}_{g}
\end{equation}
The first term in the right hand side can be estimated with
\begin{eqnarray}
\int_{B^{g}_r(x)\setminus B^{g}_{r/2}(x)} \left( \|\nabla^{g} v\|_g^2-\|\nabla^{g} u\|_g^2\right)  \,  \textnormal{dvol}_{g}&\leq& 
\int_{B^{g}_r(x)\setminus B^{g}_{r/2}(x)}2g(\nabla^{g} v, \nabla^{g}(v-u) )\,  \textnormal{dvol}_{g}\\
&=&2\int_{(B^{g}_r(x)\setminus B^{g}_{r/2}(x))\cap\{u>w\}}g(\nabla^{g}(w-u), \nabla^{g} v ) \,  \textnormal{dvol}_{g}\notag\\
&\leq&-2\beta\int_{(B^{g}_r(x)\setminus B^{g}_{r/2}(x))\cap\{u>w\}}(w-u)  \,  \textnormal{dvol}_{g}\\
& & \qquad  +\int_{\partial B^{g}_{r/2}(x)\cap\{u>w\}}\partial_{n}w (w-u) \,  \left.\textnormal{dvol}_{g}\right|_{\partial B^g_{r/2}(x)}
\end{eqnarray}
where $\partial_n$ denotes the normal derivative about $\partial B^{g}_{r/2}(x)$ (with respect to the metric $g$). We deal with the two other terms with 
$$\int_{B^{g}_r(x)\setminus B^{g}_{r/2}(x)}(u^2-v^2)  \,  \textnormal{dvol}_{g} \leq 2\|u\|_{\infty,B^{g}_r(x)}\int_{(B^{g}_r(x)\setminus B^{g}_{r/2}(x))\cap\{u>w\}}(u-w) \,  \textnormal{dvol}_{g}$$ and  $$\int_{B^{g}_{r/2}(x)} u^2  \,  \textnormal{dvol}_{g}\leq  \|u\|_{\infty,B^{g}_r(x)}^2\, \vol_g(\Om_{u}\cap B^{g}_{r/2}(x)).$$
Choosing $\beta=\lambda(m)\|u\|_{\infty,B^{g}_r(x)}$ so that two terms cancel, and denoting $$E(u,r)=\int_{B^{g}_{r/2}(x)}\|\nabla^{g} u\|_g^2\,  \textnormal{dvol}_{g}+ \mu_-(h) \vol_g(\Om_u\cap B^{g}_{r/2}(x))$$ we obtain
\begin{equation}\label{eq:E1}
E(u,r)\leq \int_{\partial B^{g}_{r/2}(x)}(\partial_{n}w) u \,  \textnormal{dvol}_{g|_{\partial B^{g}_{r/2}(x)}}+\lambda(m)\|u\|_{\infty,B^{g}_r(x)}^2\, \vol_g(\Om_{u}\cap B^{g}_{r/2}(x)).
\end{equation}
Using classical elliptic regularity results (\cite[Theorem 9.11 and 9.15]{GT01Ell}) and a scaling argument as in the proof of \ref{cor:lip}, we obtain the existence of $C$ independant of $x,r$ such that
$$\|\nabla^{g} w\|_{\infty,B^{g}_r(x)\setminus B^{g}_{r/2}(x)}\leq C\left[\frac{\|w\|_{\infty, B^{g}_r(x)}}{r}+\beta r\right]= C\|u\|_{\infty,B^{g}_r(x)}\left[\frac{1}{r}+\lambda(m)r\right].$$
%
On the other hand, working in the normal geodesic coordinates $y$ centered at $x$ such that $B^{g}_{r/2}(x_0)$ is parametrized by $B_{r/2}(0)$, using the test function $\varphi(y):=|y|^2/r$ 
we get that there exist  $C,C',C''$ independant on $x,r$ such that 
\begin{eqnarray}
\int_{\partial B^{g}_{r/2}(x_0)}u  \,  \textnormal{dvol}_{g|_{\partial B^{g}_{r/2}(x_0)}}\notag& = &
C\left[\int_{B^{g}_{r/2}(x_0)} g(\nabla^{g} u,\nabla^{g} \varphi)  \,  \textnormal{dvol}_{g} +\int_{B^{g}_{r/2}(x_0)}(\Delta_{g} \varphi)u  \,  \textnormal{dvol}_{g} \right]\\
&\leq& C' \left( \int_{B^{g}_{r/2}(x_0)}\|\nabla^{g} u\|_g  \,  \textnormal{dvol}_{g} + \frac{1}{r}\int_{B^{g}_{r/2}(x_0)}u  \,  \textnormal{dvol}_{g}
\right)\notag\\
& 
 \le &
C'\Bigg[\int_{B^{g}_{r/2}(x_0)\cap\Om_{u}}\left(\frac{1}{2} \|\nabla^{g} u\|_g^2 +\frac{1}{2}\right)  \,  \textnormal{dvol}_{g}
+ \frac{\|u\|_{\infty,B^{g}_r(x_0)}}{r}\vol_g(\Om_u\cap B^{g}_{r/2}(x_0))\Bigg]
\nonumber\\
&\leq& 
C''\Bigg[\int_{B^{g}_{r/2}(x_0)}\|\nabla^{g} u\|_g^2  \,  \textnormal{dvol}_{g}
+ \left(\frac{1}{2}+\frac{\|u\|_{\infty,B^{g}_r(x_0)}}{r}\right)\vol_g(\Om_u\cap B^{g}_{r/2}(x_0))\Bigg]\label{eq:E2}
\end{eqnarray}

From Proposition \ref{prop:pen}, we can consider
 $r$ small enough such that $h=\vol_g(B^{g}_{r/2}(x_0))$ satisfies $\mu_{-}(h)\in[\Lambda/2,\Lambda]$  and therefore, using Proposition \ref{prop:positivity} asserting that $\Lambda>0$,  \eqref{eq:E1} and \eqref{eq:E2} leads to
$$E(u,r)\leq \left(\frac{\|u\|_{\infty,B^{g}_r(x_0)}}{r}\right)\left[C''(1+\lambda(m)r^2) \left[1+ \frac{2}{\Lambda}\left(\frac{1}{2}+\frac{\|u\|_{\infty,B^{g}_r(x_0)}}{r}\right)\right]+\lambda(m)\frac{2}{\Lambda} \frac{\|u\|_{\infty,B^{g}_r(x_0)}}{r} r^2\right]E(u,r)$$
Knowing that $r\leq {r_0}$, if $\frac{\|u\|_{\infty,B^{g}_r(x_0)}}{r}$ is small enough then $E(u,r)$ must vanish, and so does $u$ in $B^{g}_{r/2}(x_0)$.
 \qed

\subsection{Finite perimeter}\label{ssect:perimeter}

As in the previous section we will use the penalization from below to prove the following result, whose proof is inspired by \cite[Lemma 5.21]{RTV18Exi}:

\begin{prop}
Let $M$ be compact and $u$ solution of \eqref{eq:pbfonc}. Then $\Om_{u}$ has finite perimeter.
\end{prop}

As in Remark \ref{rk:noncompact}, if $M$ is noncompact then one obtains that $\Om_{u}$ has locally finite perimeter.

\begin{proof}
As the manifold $M$ is compact, given $x_{0}\in\partial \Om_{u}$, it is enough to prove that $\Om_{u}$ has a finite perimeter inside a ball $B^{g}_{r}(x_{0})$ for which, applying Popositions \ref{prop:pen} and \ref{prop:positivity},  $\mu_{-}:=\mu_{-}(x_{0},r_{0},\vol_g(B^{g}_{r}(x_{0})))\geq \frac{\Lambda}{2}>0$. We want to apply the same test function as in the proof of Proposition \ref{prop:pen1}, but one needs to localize the argument. We introduce a cut-off function $\eta:M\to\R$ such that $\eta=1$ in $B^{g}_{r/2}(x_{0})$, $\eta=0$ outside $B^{g}_{r}(x_{0})$, $0\leq \eta\leq 1$, and $\|\nabla^{g}\eta\|_{g}\leq C/r$. Then one can consider $u_{t}=\eta(u-t)_{+}+(1-\eta)u$ for $t>0$, which is such that $\Om_{u_{t}}\subset \Om_{u}$ and  $\vol_g(\Om_{u_{t}})\geq \vol_g(\Om_{u})-\vol_g(B^{g}_{r}(x_{0}))$. Therefore from \eqref{eq:mu-} one has
\begin{equation}\label{eq:finiteperimeter}
\int_M \|\nabla^{g}u\|_g^2-\lambda(m)\int_M u^2+\mu_{-}\vol_g(\Om_{u})\leq \int_M \|\nabla^{g}u_{t}\|_g^2-\lambda(m)\int_M u_{t}^2+\mu_{-}\vol_g(\Om_{u_{t}}).
\end{equation}
We easily obtain the following estimates: there exists a constant $C$ depending on $u$ and $r_{0}$ such that for every $t\in (0,1]$,
\begin{itemize}
\item $\displaystyle{\int_M (u^2-u_{t}^2)\leq Ct}$
\item $\displaystyle{\int_{\{u>t\}}\left(\|\nabla_{g}u\|_g^2-\|\nabla_{g}u_{t}\|_g^2\right) \geq -Ct}$
\item $\displaystyle{\int_{\left[B^{g}_{r}(x_{0})\setminus B^{g}_{r}(x_{0})\right]\cap\{u< t\}}\left(\|\nabla_{g}u\|_g^2-\|\nabla_{g}u_{t}\|_g^2\right)\geq -Ct}$
\end{itemize}
Also, one has $\vol_g(\Om_{u})-\vol_g(\Om_{u_{t}})=\vol_g(\{0<u<t\}\cap B^{g}_{r/2}(x_{0}))$, therefore, \eqref{eq:finiteperimeter} now leads to
$$Ct\geq \int_{\{0<u<t\}\cap B^{g}_{r/2}(x_{0})}\left[\|\nabla_{g}u\|_g^2+\mu_{-}\right]\geq 2\sqrt{\mu_{-}}\int_{\{0<u<t\}\cap B^{g}_{r/2}(x_{0})}\|\nabla_{g}u\|_g.$$
Applying the co-area formula, we obtain as in the proof of Proposition \ref{prop:pen1} that $P(\Om_{u}, B^{g}_{r/2}(x_{0}))\leq \frac{C}{2\sqrt{\mu_{-}}}$, which ends the proof by compactness.\qed
\end{proof}

\subsection{Density estimates}\label{ssect:density}

From the previous results, we can obtain a first weak regularity result:

\begin{prop}\label{prop:density}
Let $u$ be a solution of \eqref{eq:pbfonc}. Then there exist $\delta>0$ and $r_{0}>0$ such that
$$\delta\leq \frac{\vol_g(\Om_{u}\cap B^{ g}_r(x_0))}{\vol_g(B^{ g}_r(x_0))}\leq 1-\delta, \;\;\;\textrm{ for any }x_{0}\in\partial\Om_{u}\textrm{ and any }r<r_{0}.$$
 \end{prop}
 
 \begin{proof}
The proof is classical, see also \cite{AC81Exi,MTV17Reg}: by Lemma \ref{lem:nondeg} asserting a non-degeneracy property for $u$, there exists $x_{r}\in B^{ g}_{r/2}(x_{0})$ such that $u(x_{r})\geq \frac{c}{4}r$ where $c$ is given in Lemma \ref{lem:nondeg}. Using now Lipschitz continuity of $u$ (Corollary \ref{cor:lip}), we get that $u>0$ on $B^{ g}_{\theta r}(x_{r})$ for some $\theta$ which does not depend on $x_{0}$; this leads to 
$$ \frac{\vol_g(\Om_{u}\cap B^{g}_{r}(x_{0}))}{\vol_g(B^g_{r}(x_{0}))}\geq \frac{\vol_g(B^{g}_{\theta r}(x_{r}))}{\vol_g(B^g_{r}(x_{0}))}. $$
hence the lower estimate.
For the upper bound, we go back to the proof of Lemma \ref{lem:lip}, where we obtained the following estimate \eqref{eq:estimate1} ($v$ being defined by \eqref{eq:v1}, and $\mu^*$ introduced in Proposition \ref{prop:pen1}):
$$ \vol_g(\{u=0\}\cap B^{ g}_r(x_0))\geq \frac{1}{\mu^*}\int_{B^{ g}_r(x_0)}\|\nabla^{ g} (u-v)\|_{ g}^2\,  \textnormal{dvol}_{ g}\,.$$
As $u=v$ on $\partial B^{ g}_r(x_0)$, we also have:
$$\int_{B^{ g}_r(x_0)}\|\nabla^{ g} (u-v)\|_{ g}^2\,  \textnormal{dvol}_{ g}\geq \lambda_{1}(B^{ g}_r(x_0))\int_{B^{ g}_r(x_0)}(u-v)^2\,  \textnormal{dvol}_{ g}\,.$$
We want to obtain a lower bound for this last term, which will rely on the fact that $u$ is small near $x_{0}$, while $v$ is large. By Lipschitz continuity, we have $u(x)\leq L\kappa r$ in $B^{ g}_{\kappa r}(x_{0})$, where $L$ is Lipschitz constant for $u$. On the other hand, we had shown that
$$v\geq c_{6}\left(\frac{1}{r}\fint_{\partial B^{ g}_r(x_0)}u \,  \textnormal{dvol}_{ g|_{\partial B^{ g}_{r}(x_0)}}\right)(r-d_{ g}(x,x_0)).$$
From Corollary \ref{cor:nondeg} one has $\left(\frac{1}{r}\fint_{\partial B^{ g}_r(x_0)}u\,  \textnormal{dvol}_{ g|_{\partial B^{ g}_{r}(x_0)}}\right)\geq c$
and therefore $v\geq c_{7}(1-\kappa)r$ on $B^{ g}_{\kappa r}(x_{0})$ for some constant $c_{7}>0$. Choosing $\kappa$ small enough, this leads to 
$v-u\geq c_{8}r$ in $B^{ g}_{\kappa r}(x_0)$ with $c_{8}>0$, which leads to
$$\vol_g(\{u=0\}\cap B^{ g}_r(x_0))\geq cr^2\lambda_{1}(B^{ g}_r(x_0))\vol_g(B^{ g}_{\kappa r}(x_0)) $$
and allows to conclude from the facts that $r^2\lambda_{1}(B^{ g}_r(x_0))$ is uniformly bounded from below for $r$ small (see \cite{Druet}), and  $\vol_g(B^{ g}_{\kappa r}(x_0))\geq c\vol_g(B^{ g}_r(x_0))$.
\qed
\end{proof}

\subsection{Weiss-monotonicity formula in a manifold}

Let $u$ be a local minimum of \eqref{eq:pbfonc} and $x_{0}\in \partial\Om_{u}$. Following \cite{W98Par}, we define, for $r$ small enough, the function
\begin{equation}\label{phiweiss}
\phi_{u,x_0}^g(r):=\frac{1}{r^{n}}\int_{B^{ g}_{r}(x_{0})} \left(\|\nabla^{ g} u\|_{ g}^2+\Lambda \mathbbm{1}_{u>0}\right) \,  \textnormal{dvol}_{ g}  -\frac{1}{r^{n+1}}\int_{\partial B^{ g}_{r}(x_{0})}u^2 \,  \textnormal{dvol}_{ g|_{\partial B_r^{ g}(x_0)}}  
\end{equation} 

We have the following:
\begin{prop}\label{prop:monotonicity}
There exists $C>0$ and $r_{0}>0$ such that for all $r<r_{0}$,
\begin{equation}\label{monot}
(\phi^g_{u,x_{0}})'(r)\geq \frac{2}{r^{n}}\int_{\partial B^{ g}_{r}(x_0)}\left(\partial_{\nu} u-\frac{u}{r}\right)^2\,  \textnormal{dvol}_{ g|_{\partial B_r^{ g}(x_0)}} {-Cr} \,.
\end{equation}
\end{prop}

We give immediately the following corollary of Proposition \ref{prop:monotonicity}, which will be fundamental in the blow-up procedure of the following subsection:

\begin{corollary}\label{limit-weiss}
The limit $\displaystyle \lim_{r \to 0^+} \phi_{u,x_0}^g(r)$ exists and is finite.
\end{corollary}

\begin{proof} It suffices to observe firstly that the function $\phi_{u,x_0}^g(r)$ is bounded for $r$ small because $u$ is Lipschitz continuous (\ref{cor:lip}), and secondly that the function
\[
r \to \phi_{u,x_0}^g(r) + \frac{C}{2}\,r^2\, ,
\]
is monotone nondecreasing by Proposition \ref{prop:monotonicity}.
\qed
\end{proof}

\medskip

We turn now to the proof of Proposition \ref{prop:monotonicity}. We cannot proceed as in \cite{MTV17Reg}, as if we use optimality, one cannot control the penalization term with enough precision. Therefore, we use the approach from \cite{W98Par} using only the Euler-Lagrange equation from Lemma \ref{lem:euler}, which is available with volume constraint. See also \cite{RTV18Exi} for a similar strategy. 

\begin{remark}\label{rk:weiss}
In \cite{W98Par}, it is proven, when $M$ is replaced by $D$ a bounded domain in $\R^n$ and when $g$ is the euclidian metric, that 
\begin{equation}\label{eq:monotonyAC}
(\phi^e_{u,0})'(r)\, =\, \frac{2}{r^{n}}\int_{\partial B_{r}(0)}\left(\partial_{r} u-\frac{u}{r}\right)^2 \, \textnormal{dvol}_{e} \,,
\end{equation}
if $u$ is such that
\begin{equation}\label{eq:ELAC}\forall \Phi\in C^{\infty}_{c}(D,\R^n),\;\;
\int_{D} \Big[2\, (D\Phi\nabla^{e} u)\cdot \nabla^{e} u-  |\nabla^{e} u|^2 \, \div \Phi\Big]\textnormal{dvol}_{e}= \Lambda\int_{\Om_u\cap D}\div \Phi\ \,\textnormal{dvol}_{e}. 
\end{equation}
which is the Euler-Lagrange equation (similarly to Lemma \ref{lem:euler}) for the minimization of the Alt-Caffarelli functional $\displaystyle{u\mapsto \int_{D}|\nabla^eu|^2\textnormal{dvol}_{e}+\Lambda\vol_{e}(\Om_{u}\cap D)}$. We will use this result when studying blow-up limits in the next section.
\end{remark}

\noindent{\bf Proof of Proposition \ref{prop:monotonicity}.} We proceed in several steps.

\medskip

{\bf First step.} In this first step we just compute formally the derivative of the function $\phi^g_{u,x_0}$ with respect to $r$, without using the optimality of $u$. If we denote 
 $$A(r)=\int_{B^{ g}_{r}(x_{0})} \left(\|\nabla^{ g} u\|_{ g}^2+\Lambda \mathbbm{1}_{u>0}\right) \,  \textnormal{dvol}_{ g}\,, $$ 
we classically have
$$A'(r)=\int_{\partial B^{ g}_{r}(x_0)}\left(\|\nabla^{ g} u\|_{ g}^2+\Lambda\mathbbm{1}_{u>0} \right)\,  \textnormal{dvol}_{ g|_{\partial B_r^{ g}(x_0)}}\,, $$
and then we can write:
\begin{equation}\label{derphi}
\phi'_{u,x_{0}}(r)=-\frac{n}{r^{n+1}}A(r)+\frac{1}{r^{n}}A'(r)-\frac{d}{dr}\left[\frac{1}{r^{n+1}}\int_{\partial B^{ g}_{r}(x_{0})}u^2 \,  \textnormal{dvol}_{ g|_{\partial B_r^{ g}(x_0)}}  \right]\,.
\end{equation}
Let us compute explicitly the third terms of \eqref{derphi}:
$$\frac{d}{dr}\left[\frac{1}{r^{n+1}}\int_{\partial B^{ g}_{r}(x_0)}u^2 \,  \textnormal{dvol}_{ g|_{\partial B_r^{g}(x_0)}} \right]=
-\frac{n+1}{r^{n+2}}\int_{\partial B_{r}^g(x_0)}u^2\,  \textnormal{dvol}_{ g|_{\partial B_r^{g}(x_0)}}+\frac{1}{r^{n+1}}\frac{d}{dr}\left[\int_{\partial B^g_{r}(x_0)}u^2\,  \textnormal{dvol}_{ g|_{\partial B_r^{g}(x_0)}}\right]$$ 
and we have
\begin{eqnarray*}
\frac{d}{dr}\left[\int_{\partial B^g_{r}(x_0)}u^2\,  \textnormal{dvol}_{ g|_{\partial B_r^{g}(x_0)}}\right]&=&\int_{\partial B^g_{r}(x_0)} g(\nabla(u^2),\nu) \,  \textnormal{dvol}_{ g|_{\partial B_r^{g}(x_0)}}+\int_{\partial B^g_{r}(x_{0})}u^2\, H\,  \textnormal{dvol}_{ g|_{\partial B_r^{g}(x_0)}}\\
& = & \int_{\partial B^g_{r}(x_0)} 2u\partial_\nu u \,  \textnormal{dvol}_{ g|_{\partial B_r^{g}(x_0)}}+\int_{\partial B^g_{r}(x_{0})}u^2\, H\,  \textnormal{dvol}_{ g|_{\partial B_r^{g}(x_0)}}
\end{eqnarray*}
where $\nu$ and $H$ denote the outer normal vector to $\partial B_{r}^g(x_{0})$ and the mean curvature (sum of the principal curvatures) respectively, and $\partial_\nu$ denotes the normal derivative with respect to the metric $g$. In conclusion, the third term of the second member of \eqref{derphi} is given by
$$
\frac{d}{dr}\left[\frac{1}{r^{n+1}}\int_{\partial B^g_{r}(x_{0})}u^2 \,  \textnormal{dvol}_{ g|_{\partial B_r^{g}(x_0)}} \right]=\frac{1}{r^{n+1}}\int_{\partial B_{r}^g(x_0)}\left[\left(H-\frac{n+1}{r}\right)u^2+2u\partial_{\nu} u\right] \,  \textnormal{dvol}_{ g|_{\partial B_r^{g}(x_0)}}\,.$$
Finally we obtain:
\begin{equation}
\phi_{u,x_0}'(r)=\frac{1}{r^{n+1}}\left\{-n A(r)+r A'(r)-\int_{\partial B^{ g}_{r}(x_0)}\left[\left(H-\frac{n+1}{r}\right)u^2+2u\partial_{\nu} u\right] \,  \textnormal{dvol}_{ g|_{\partial B_r^{g}(x_0)}}\right\}\,.
\end{equation}
Using the fact that on $\partial B_r^g(x_0)$ the mean curvature is
\[
H = \frac{n-1}{r} + \mathcal{O}(r)
\]
(where the error term $\mathcal{O}(r)$ is smooth in $r$) we obtain:
\begin{equation}
\phi_{u,x_0}'(r)=\frac{1}{r^{n+1}}\left\{-n A(r)+r A'(r)-\int_{\partial B^{ g}_{r}(x_0)}\left[2u\partial_{\nu}u+\left(-\frac{2}{r} + \mathcal{O}(r) \right) \,u^2\right]\textnormal{dvol}_{ g|_{\partial B_r^{ g}(x_0)}}\right\}
\end{equation}

\medskip

{\bf Second step: } We want now to compute the main order of the term
\[
-n A(r)+r A'(r)\,,
\]
and in this computation we will use the optimality condition for the function $u$. We claim that
\begin{multline}\label{nj}
-n A(r)+r A'(r) = (1+\mathcal{O}(r))(n+2)\lambda(m)\int_{B^{ g}_{r}(x_0)}  u^2 \textnormal{dvol}_{ g} +\\
+ (1+\mathcal{O}(r)) \int_{\partial B^{ g}_{r}(x_0)} \left[2u\partial_\nu u -2r(\partial_{\nu}u)^2 -\lambda(m) r u^2\right]\textnormal{dvol}_{ g|_{\partial B_r^{ g}(x_0)}}\,,
\end{multline}
where the error term $\mathcal{O}(r)$ is smooth in $r$. In order to prove the claim, let $x = (x_1, ..., x_n)$ be normal geodesic coordinates aroung $x_0$ as in \eqref{eq:theta}. Let $\Phi_{\eps}$ be an approximation of 
\[
\Phi(\textnormal{exp}_{x_0}\Theta (x)) =\Theta(x) \mathbbm{1}_{B^{ g}_{r}(x_0)}\,.
\]
More precisely, given $\eps>0$, we consider 
\[
\Phi_{\eps}(\textnormal{exp}_{x_0}\Theta (x)) =\rho_{\eps}(\textnormal{exp}_{x_0}\Theta (x))\, \Theta(x)
\] 
where $\rho_{\eps}(\textnormal{exp}_{x_0}\Theta (x))=\varphi_{\eps}(d(\textnormal{exp}_{x_0}\Theta (x), x_0))$ is smooth and equal to 1 in $B_{r}^g(x_{0})$, vanishes outside $B_{r+\eps}^g(x_{0})$, and such that $\varphi_{\eps}$ is decreasing on $[r,r+\eps]$. Using Lemma \ref{lem:euler} we obtain
\begin{multline*}
\int_{M}\left[ 2\rho_{\eps}\|\nabla^{g} u\|_{g}^2+\big(\lambda(m)u^2-\|\nabla^{g} u\|_{g}^2\big)\, \textnormal{div}_g \Theta(x)\, \rho_{\eps}+2g((\Theta(x)\otimes\nabla^g\rho_{\eps})\nabla^g u,\nabla^g u)+ \right.\\+ \left.\big(\lambda(m)u^2-\|\nabla^{g} u\|_{g}^2\big)g(\Theta(x),\nabla^g \rho_{\eps})\right] \,  \textnormal{dvol}_{ g} =\Lambda\int_{\Om_{u}} \left( \textnormal{div}_g \Theta(x)\, \rho_{\eps}+g(\Theta(x),\nabla^g \rho_{\eps})\right)\,  \textnormal{dvol}_{g}
\end{multline*}
Passing to the limit $\eps\to 0$, we obtain: 
\begin{multline}\label{multi}
\int_{B^g_{r}(x_0)}\left[2\|\nabla^{g} u\|_{g}^2+ \textnormal{div}_g \Theta(x)\, \big(\lambda(m)u^2-\|\nabla^{g} u\|_{g}^2\big)\right]  \,  \textnormal{dvol}_{ g}
+\\-(r+\mathcal{O}(r^2)) \int_{\partial B^g_{r}(x_0)}\Big[2(\partial_{\nu}u)^2+\lambda(m)u^2-\|\nabla^{g} u\|_{g}^2\Big] \,  \textnormal{dvol}_{ g|_{\partial B_r^{g}(x_0)}}=\\=\Lambda \left(\int_{B_{r}^g(x_0)}  \textnormal{div}_g \Theta(x)\, \mathbbm{1}_{u>0} \,  \textnormal{dvol}_{ g}-(r+\mathcal{O}(r^2))\int_{\partial B^g_{r}(x_0)}\mathbbm{1}_{u>0}\,  \textnormal{dvol}_{ g|_{\partial B_r^{g}(x_0)}}\right)\,.
\end{multline}
In fact, in order to justify the previous passage, we use the fact that, since $r$ is small, the metric $g$ can be approximated with the Euclidean one. Then, first we have
\begin{multline*}
\int_{M}g((\Theta(x)\otimes \nabla^g \rho_{\eps})\nabla^g u,\nabla^g u)\,  \textnormal{dvol}_{ g}=
(1+\mathcal{O}(r)) \int_{B_{r+\eps}\setminus B_{r}}\varphi_{\eps}'(|x|)\frac{x_{i}x_{j}\partial_{i}u\partial_{j}u}{|x|} \, dx \\ \mathop{\longrightarrow}_{\eps\to 0} -(r+\mathcal{O}(r^2))\int_{\partial B_{r}}\frac{x_{i}x_{j}\partial_{i}u\partial_{j}u}{|x|^2} \, d\sigma_x =-(r+\mathcal{O}(r^2))\int_{\partial B_{r}}(\partial_{r}u)^2 \, d\sigma_x = -(r+\mathcal{O}(r^2))\int_{\partial B^g_{r}(x_0)}(\partial_{\nu}u)^2 \, \textnormal{dvol}_{ g|_{\partial B_r^{g}(x_0)}}
\end{multline*}
where $B_r$ represents the Euclidean ball of radius $r$, $|x|$ the Euclidean lengh of $x$, $dx, d\theta, ds$ the Euclidean Lebesgue measure with respect to the variables $x,\theta,s$, and $d\sigma_x$ the Euclidean Lebesgue measure induced on $\partial B_r$. Secondly, for a general function $f$
\begin{multline*}
\int_{M}f(x) g(\Theta(x),  \nabla^g \rho_{\eps})\,  \textnormal{dvol}_{ g} = 
(1+\mathcal{O}(r)) \int_{B_{r+\eps}\setminus B_{r}} f(x)\varphi_{\eps}'(|x|) |x|\, dx= (1+\mathcal{O}(r)) \int_{r}^{r+\eps}s^n\varphi_{\eps}'(s)\int_{\S^{n-1}} f(r\theta) d\theta ds\\\mathop{\longrightarrow}_{\eps\to 0}-(1+\mathcal{O}(r))r^n\int_{\S^{n-1}}f(r\theta)d\theta=-(r+\mathcal{O}(r^2))\int_{\partial B_{r}}f \, d\sigma_x =-(r+\mathcal{O}(r^2)) \int_{\partial B^g_{r}(x_0)}f\,  \textnormal{dvol}_{ g|_{\partial B_r^{g}(x_0)}}
\end{multline*}
where we used the same notation as before, and $d\theta, ds$ are the Euclidean Lebesgue measures with respect to the spherical variables $\theta,s$.
We observe that
\[
\textnormal{div}_g \Theta(x) = \frac{1}{\sqrt{\det g}} \frac{\partial_i \left( \frac{\sqrt{\det g}}{g_{ii}} x_i \right)}{\partial x_i} = n + \mathcal{O}(r^2)
\]
and then \eqref{multi} can be rewritten as
\begin{multline*}(1+\mathcal{O}(r^2))\,  (n\,A(r)-r\,A'(r))=2\int_{B^{ g}_{r}(x_0)}\|\nabla^g u\|_g^2\, \textnormal{dvol}_{ g}-2(r+\mathcal{O}(r^2))\int_{\partial B^{ g}_{r}(x_0)}(\partial_{\nu}u)^2 \textnormal{dvol}_{ g|_{\partial B_r^{ g}(x_0)}}\\
+\lambda(m) (n+\mathcal{O}(r^2)) \int_{B^{ g}_{r}(x_0)}u^2\textnormal{dvol}_{ g}-\lambda(m) (r+\mathcal{O}(r^2)) \int_{\partial B^{ g}_{r}(x_0)}u^2 \textnormal{dvol}_{ g|_{\partial B_r^{ g}(x_0)}}.
\end{multline*}
Using the fact the following easy consequence of the Green formula
\[
\displaystyle{\int_{B^{ g}_{r}(x_0)}\|\nabla^g u\|_g^2=\lambda(m)\int_{B^{ g}_{r}(x_0)}u^2\textnormal{dvol}_{ g}+\int_{\partial B^{ g}_{r}(x_0)}u\partial_{\nu}u},
\] 
 we finally obtain \eqref{nj}, and the claim is proved.

\medskip

{\bf Third step: } In this last step we use the boundary condition $u(x_{0})=0$ and the Lipschitz continuous regularity of the optimal function $u$, which has been proven in Corollary \ref{cor:lip}.\\
Replacing \eqref{nj} in \eqref{derphi} we obtain: 
\begin{eqnarray*}\phi_{u,x_0}'(r)
&=&\frac{(1+\mathcal{O}(r))}{r^{n+1}}\left[-(n+2)\lambda(m) \int_{B^{ g}_{r}(x_0)}u^2\textnormal{dvol}_{ g}+\right.\\
& & \qquad \qquad \qquad \qquad + \left. \int_{\partial B^{ g}_{r}(x_0)}\left(
2r(\partial_{\nu}u)^2 +\lambda(m) r u^2 -4u\partial_{\nu}u+2\frac{u^2}{r}\right)\textnormal{dvol}_{ g|_{\partial B_r^{ g}(x_0)}}\right.\\
&=&
(1+\mathcal{O}(r)) \left[ \frac{2}{r^{n}}\int_{\partial B^{ g}_{r}(x_0)}\left(\partial_{\nu} u-\frac{u}{r}\right)^2\textnormal{dvol}_{ g|_{\partial B_r^{ g}(x_0)}}+ \right.\\
& & \qquad \qquad \qquad \qquad \left. -\frac{(n+2)\lambda(m)}{r^{n+1}}\int_{B^{ g}_{r}(x_0)}u^2\textnormal{dvol}_{ g}+\frac{\lambda(m)}{r^n}\int_{\partial B^{ g}_{r}(x_0)}u^2\textnormal{dvol}_{ g|_{\partial B_r^{ g}(x_0)}} \right]
\end{eqnarray*}
Proposition \ref{prop:monotonicity} follows then from the facts that $u(x_{0})=0$ and that $u$ is Lipschitz continuous.
\qed

\subsection{Blow-up procedure}\label{ssect:blowup}

Let $u$ be a solution of \eqref{eq:pbfonc}, $x_{0}\in\partial\Om_{u}$. 
 In this section we want to build and study the blow up of the solution $u$ around $x_0$. To that end, given $\eps>0$ small enough, we define on the manifold $M$ the metric $\bar g : =  \eps^{-2} \, g$ as we did at the beginning of this section, and the parameterization of $B^{g}_{\sqrt{\eps}}(x_0)$ given by
\[
Y ( y) : =\mbox{exp}_{x_0}^{g} \left(\eps \,  \Theta (y) \right)
\]
with $y = (y_1,...,y_n) \in B_{\frac{1}{\sqrt{\eps}}}:=\left\{ y \in \R^n \, /\, |y|<\frac{1}{\sqrt{\eps}} \right\}$ and $\Theta$ defined in \eqref{eq:theta}. 
In $B_{\sqrt{\eps}}^g (x_0)$ the metric $\bar g$ is the Euclidean one up to $\eps$ terms. Let us consider $B_{\frac{1}{\sqrt{\eps}}}$ with the metric $\bar g$ (induced by the parameterization $Y$). 
In the coordinates $y$ beloning to $B_{\frac{1}{\sqrt{\eps}}}$ we define 
\[
u_\eps(y):=\frac{1}{\eps}\, u( \textnormal{exp}^g_{x_0}(\eps\,\Theta(y)))\,.
\]
We have the following result:
 
 \begin{prop}\label{prop:blowupexist}
Let $u$ be a solution of \eqref{eq:pbfonc}, $x_{0}\in \partial\Om_{u}$. Then there exists a sequence $\eps_{k}$ going to 0 such that $u_{\eps_{k}}$ converges to a function $u_{0}:\R^n\to\R$ uniformly on any compact set. Moreover $u_{0}$ is nonnegative and Lipschitz continuous.
\end{prop}

We call such $u_{0}$ a blow-up limit of $u$ at $x_{0}$.

\medskip

\begin{proof}
Let $R<\frac{1}{\sqrt{\eps}}$. With the notation we introduced above, we have
\[
\nabla^{\bar g} u_\eps(y)=\nabla^{g} u(\mbox{exp}_{x_0}^{g} \left(\eps \,  \Theta (y) \right)
)
\]
and
\[
\|\nabla^{\bar g} u_{\eps}\|_{\infty,B_{R}(0)}\leq \|\nabla^{g} u\|_{\infty,B^{g}_{\eps R}(x_{0})}
\] 
and so the Euclidean gradient of $u_\eps$ is uniformly bounded in $B_{R}(0)$. Therefore, as $u(x_{0})=0$, we also get for any $y\in B_{R}$, 
\[
|u_{\eps}(y)|\leq C\, \|\nabla^{\bar g} u_{\eps}(y)\|_{\bar g}\,\|y\|_{\bar g} \leq C'\,\|\nabla^g u\|_{\infty,B^g_{\eps\,R}(x_{0})}
\] 
which is also bounded uniformly in $\eps$. From Arzel\`a-Ascoli Theorem, we deduce that up to a subsequence, $u_{\eps}$ converges uniformly to a function $u_{0}$ on $B_{R}(0)$. Using a diagonalization argument, we prove that up to a subsequence, $u_{\eps}$ converges to $u_{0}: \mathbb{R}^n \to \mathbb{R}$ uniformly on every ball $B_{R}(0)$ and therefore on any compact set. The properties of $u_{0}$ follow easily.\qed
\end{proof}



Using the previous subsections, we are also in position to prove, similarly to \cite{W98Par}, the following properties of blow ups:

\begin{prop}\label{prop:blowuphomog}
Let $u$ be a solution of \eqref{eq:pbfonc}, $x_{0}\in \partial\Om_{u}$, and $u_{0}$ a blow-up limits of $u$ at $x_{0}$. Then $u_0$ is (positively) $1$-homogeneous and is a non-trivial
 global solution of the Alt-Caffarelli functional, which means:
\begin{equation}\label{eq:pbblowup}
\int_{B_{R}(0)}|\nabla^e u_0|^2 + \Lambda\, \vol_e(\Om_{u_{0}}\cap B_{R}(0)) \le\int_{B_{R}(0)} |\nabla^e w|^2 + \Lambda
\, \vol_e(\Om_{w}\cap B_{R}(0)).
\end{equation}
for every $R>0$ and $w\in H^1_{loc}(\R^n)$ such that $w=u_0$ outside $B_{R}(0)$.
\end{prop}


\begin{proof}{ 
 By definition, there exists $\eps_{k}$ going to 0 such that $u_{\eps_{k}}$ converges to $u_{0}$ on any compact set in $\R^n$. For the sake of clarity, we drop reference to the subsequence and denote $\eps$ instead of $\eps_{k}$. We proceed by steps.
 
 \medskip

{\bf First step: }we start improving the convergence. Following the proof of \cite[Proposition 4.5(a)]{MTV17Reg} we obtain that for any $R>0$, $u_\eps$ converges to $u_{0}$ strongly in $H^1(B_{R}(0))$, and that $\Om_\eps:=\{u_\eps>0\}$ converges to $\Om_{0}:=\{u_{0}>0\}$ strongly in $L ^1(B_{R})$.

\medskip

{\bf Second step: }using the density estimate given in Proposition \ref{prop:density}, it is classical that we have in fact  convergence of $\overline{\Om_{\eps}}$ to $\overline{\Om_{0}}$ and $\Om_{\eps}^c$ to $\Om_{0}^c$ for the Hausdorff metric in $B_{R}$ (see for example \cite[Proof of Theorem 2]{IM18Inf}).

\medskip

{\bf Third step: }we now prove that $u_{0}$ is a global minimizer of the Alt-Caffarelli functional. Let $R>0$ and $w\in H^1_{loc}(\R^n)$ such that $w=u_0$ outside $B_{R}(0)$. For any given $v$ defined at least on $B_{\frac{1}{\sqrt{\eps}}}$ of $\R^n$ for some $\eps$, we define for $x < \sqrt{\eps}$,
\[
v^{\eps}(\textnormal{exp}^{g}_{x_0}(\Theta(x)))=\eps\, v\left(\frac{x}{\eps}\right)
\] 
the blow-down of $v$ centered at $x_0$, assuming that the injectivity radius at $x_0$ is bigger than $\sqrt{\eps}$. Moreover, by scaling properties,
$$J(v^{\eps})=\int_{B_{\sqrt{\eps}}^{g}(x_0)}\|\nabla^{g} v^{\eps}\|_{g}^2 \, \textnormal{dvol}_{g} +\lambda(m)\int_{B_{\sqrt{\eps}}^{g}(x_0)}(v^{\eps})^2 \,  \textnormal{dvol}_{g} =\eps^n\left(\int_{B_{\frac{1}{\sqrt{\eps}}}} \|\nabla^{\bar g} v\|_{\bar g}^2 \,  \textnormal{dvol}_{\bar g} +\lambda(m)\, \eps^2\int_{B_{\frac{1}{\sqrt{\eps}}}} v^2 \,  \textnormal{dvol}_{\bar g} \right)$$
and $\vol_g(\Om_{v^{\eps}})=\eps^n\,\vol_{\bar g}(\Om_{v})$.
 We introduce $\eta\in C^\infty_{c}(B_{R}(0))$ such that $0\leq \eta\leq 1$, and define 
 \[
 w_{\eps}:=w+(1-\eta)(u_{\eps}-u_{0})
 \] 
 which is equal to $u_{\eps}$ outside $B_{R}(0)$. Therefore
$(w_\eps)^{\eps}=u$ in $B_{\sqrt{\eps}}^g(x_0) \setminus B_{\eps\, R}^g (x_0)$, and therefore 
\[
h_\eps:=\big|{\vol_g(\Om_{(w_\eps)^\eps})-\vol_g(\Om_{u})}\big|\leq C\,\eps^n
\] 
so by Proposition \ref{prop:pen}, we have for $\eps$ small enough
$$J(u)+\mu(h_\eps)\vol_g(\Om_{u})\leq J((w_\eps)^{\eps})+\mu(h_\eps)\vol_g(\Om_{(w_\eps)^{\eps}}).$$
So, using the previous scaling properties and localizing in $B_{{R}}$ we obtain
\begin{multline*}\int_{B_{{R}}}\|\nabla^{\bar g} u_{\eps}\|_{\bar g}^2 \, \textnormal{dvol}_{\bar g} +\lambda(m)\, \eps^2\int_{B_{{R}}} (u_{\eps})^2 \, \textnormal{dvol}_{\bar g} +\mu(h_\eps)\vol_{\bar g}(\Om_{u_{\eps}}\cap B_{{R}})\\
\leq
\int_{B_{{R}}}\|\nabla^{\bar g} w_\eps\|_{\bar g}^2 \, \textnormal{dvol}_{\bar g} +\lambda(m)\, \eps^2\int_{B_{{R}}} (w_\eps)^2\, \textnormal{dvol}_{\bar g} +\mu(h_\eps)\vol_{\bar g}(\Om_{w_\eps}\cap B_{{R}}).
\end{multline*}
We first use the inclusion $\Om_{w_\eps}\cap B_{{R}}\subset \{x\in B_{{R}}, w(x)>0\textrm{ and }\eta(x)=1\}\cup\{x\in B_{{R}}, 0\leq \eta(x)<1\}$ to dominate $\vol_{\bar g}(\Om_{w_\eps})$.
Then using that $u_\eps$ converges to $u_{0}$ strongly in $H^1_{loc}(\R^n)$, and as the same goes for $w_\eps$, we obtain as $\eps\to 0$ (taking in account that the metric $\bar g$ converges uniformly to the Euclidean metric):
$$\int_{B_{{R}}}|\nabla^e u_{0}|^2+\Lambda\, \vol_e(\Om_{u_{0}}\cap B_{{R}})\leq 
\int_{B_{{R}}}|\nabla^e w|^2+\Lambda(\vol_e(\Om_{w}\cap \{\eta=1\})+\vol_e(\{0\leq \eta<1\}\cap B_{{R}}).$$
We conclude by choosing $\{\eta=1\}$ arbitrary close to $B_{{R}}$. 

\medskip

{\bf Fourth step: }Let us prove now that $u_{0}$ is 1-homogeneous. with the notations of Proposition \ref{prop:monotonicity} and by scaling properties, we have
\[
\phi^{\bar g}_{u_\eps,0}(r)=\phi^g_{u,x_0}(\eps\,r)\,.
\]
From Corollary \ref{limit-weiss}, we know that $\phi^g_{u,x_0}(0^+)$ exists, and so 
\[
\lim_{\eps\to0^+}\phi^{\bar g}_{u_{\eps},0}(r)=\phi^g_{u,x_0}(0^+)\,.
\]
From the convergence properties of the blow-up that we proved above, we have on the other hand that 
\[
\lim_{\eps\to0^+}\phi^{\bar g}_{u_{\eps},0}(r)=\phi^e_{u_{0},0}(r)\,.
\] 
Combining the two previous results, we obtain that $(\phi^e_{u_{0},0})'(r)=0$, which implies from \cite{W99Par} (see Remark \ref{rk:weiss}) that $u_{0}$ is 1-homogeneous.

\medskip

{\bf Fifth step: }Classically, from the non-degeneracy of $u$ given in Lemma \ref{lem:nondeg}, we conclude that $u_{0}$ is non-trivial, see for example \cite[Proposition 4.5]{MTV17Reg}.}

\medskip 

The proof of the proposition is then complete. 
\qed
\end{proof}

\medskip

From the previous results, one can deduce a formula linking the limit at 0 of the Weiss functional $\phi^g_{u,x_{0}}(0^+)$ and the density of $x_{0}\in\partial\Om_{u}$.

\begin{prop}\label{prop:density2} (Density formula). Let $u$ be a solution of \eqref{eq:pbfonc} and $x_{0}\in \partial\Om_{u}$. For any blow-up $u_{0}$ of $u$ at $x_{0}$, we have
\begin{equation}\label{eq:density}
\theta(x_{0}):=\liminf_{r\to 0} {\frac{\vol_{g}(\Om\cap B^{g}_{r}(x_{0}))}{\vol_{g}(B^{g}_{r}(x_0))}}=\frac{1}{\Lambda\omega_{n}}\phi^g_{u,x_{0}}(0^+)=\frac{\vol_e(\{u_{0}>0\}\cap B_{R})}{\vol_e(B_{R})}\,, 
\end{equation}
where $\phi^g_{u,x_0}$ is given by \eqref{phiweiss} and $R>0$.
\end{prop}

\begin{proof} 
On one hand, because of the definition of blow ups, considering $u_{\eps}$ converging (up to a subsequence) to $u_{0}$, one has for $\eps$ small enough
$${\frac{\vol_{g}(\Om_{u}\cap B^g_{\eps}(x_{0}))}{\vol_{g}(B^g_{\eps}(x_{0}))}}=\frac{\vol_{\bar g}(\Om_{u_{\eps}}\cap B_{1})}{\vol_{\bar g}(B_{1})}$$ which gives at the limit
$\eps\to 0$, 
\[
\theta(x_{0})=\frac{\vol_e(\Om_{u_{0}}\cap B_{1})}{\vol_e(B_{1})}\,.
\] 
Moreover, $u_{0}$ being 1-homogeneous, the last term is independant on the radius of the considered ball.
On the other hand, in the fourth step of the proof of Proposition \ref{prop:blowuphomog}, we have seen that 
$\phi_{u_{0},0}^e(r)=\phi^g_{u,x_{0}}(0^+).$
As $u_{0}$ is 1-homogeneous, we also have 
\[
\phi_{u_{0},0}^e(r)=\Lambda\frac{\vol_e(\{u_{0}>0\}\cap B_{r})}{r^d}\, ,
\] 
which concludes the proof.\qed
\end{proof}

We conclude this section with the following corollary about the possible values of the density:

\begin{corollary}\label{cor:gap} (Density bound-density gap). 
Let $u$ be a solution of \eqref{eq:pbfonc} and $x_{0}\in \partial\Om_{u}$. Then $\theta(x_{0})\geq \frac{1}{2}\,,$ and there exists $\eta>0$ independant on $x_{0}$ such that if $\theta(x_{0})\neq \frac{1}{2}$, then $\theta(x_{0})\geq \frac{1}{2}+\eta$.
\end{corollary}

We do not reproduce the proof of this result, as it is of purely Euclidean nature, when combined with Proposition \ref{prop:density2}; see \cite[Lemma 5.3 and 5.4]{MTV17Reg} or \cite[Proposition 6.1 and Lemma 6.3]{DET17Fre} for a detailed proof. For the sake of clarity, let us nevertheless recall the idea behind the first part of the result, namely that $\theta(x_{0})\geq \frac12$. Thanks to Propositions \ref{prop:blowuphomog} and \ref{prop:density2}, this relies on the non-existence of a non-trivial 1-homogeneous harmonic function on a cone of density less that $\frac{1}{2}$. To see this, consider $u_{0}$ such a function, then its trace on the sphere $\S^{n-1}$ is a first eigenfunction of the Laplace-Beltrami operator on subdomain of the sphere with eigenvalue $n-1$. This cannot happen for a strict subset of the sphere, because of Faber-Krahn type results in the sphere.

The second part of the statement (the gap estimate) is more involved and relies on a ``flatness imply regularity'' result (\cite[Theorem 8.1]{AC81Exi} or \cite[Theorem 1]{DS11Fre}).

To conclude this section, we see that the optimality condition $\|\nabla^gu\|_{g}=\sqrt{\Lambda}$ is valid in the sense of viscosity. 

\begin{prop} Let $u$ be a solution of \eqref{eq:pbfonc}. The function $u$ is a viscosity solution to $\|\nabla^{g} u\|_{g}=\sqrt{\Lambda}$ on $\partial\Om_{u}$. This means that for every $x_{0}\in\partial\Om$:
\begin{itemize}
\item if $\varphi:M\to\R$ is differentiable at $x_{0}$ and such that $u\geq \varphi$ in $\Om_{u}$ with equality at $x_{0}$, then $\|\nabla^{g} \varphi\|_{g}(x_{0})\leq \sqrt{\Lambda}$,
\item if $\varphi:M\to\R$ is differentiable at $x_{0}$ and such that $u\leq \varphi$ in $\Om_{u}$ with equality at $x_{0}$, then $\|\nabla^{g} \varphi\|_{g}(x_{0})\geq \sqrt{\Lambda}$.
\end{itemize}
\end{prop}
The proof of this result follows exactly the same lines as \cite[Lemma 5.2]{MTV17Reg} and relies again on Proposition \ref{prop:blowuphomog} and the study of non-trivial 1-homogeneous global minimizers of the Alt-Caffarelli functional.

\subsection{Proof of Theorem \ref{th:reg0}}\label{ssect:conclusion}

We are now in position to prove Theorem \ref{th:reg0}.  The proof is now very close to the proof of \cite[Proposition 5.18]{MTV17Reg}, so we give fewer details; see also \cite[Propositions 5.32 and 5.35]{RTV18Exi} or \cite[Corollary 7.2 and Theorem 8.1]{DET17Fre} for more detailed proofs. 

\medskip

\noindent{\bf Proof of Theorem \ref{th:reg0}. } Let $\Om^*$ be a solution of \eqref{eq:pbforme2}. Then as $M$ is assumed to be connected, from Proposition \ref{prop:freebdy} and \ref{prop:saturation}, $\Om^*=\Om_{u}$ where $u=u_{\Om^*}$ is a solution of \eqref{eq:pbfonc}.

The first point of the Theorem has been proved in Sections \ref{ssect:lip} and \ref{ssect:perimeter}. Let us prove the second part of the statement.
\begin{itemize}
\item[(a)] Let us define $\Sigma_{reg}:=\{x_{0}\in\partial\Om^*, \theta(x_{0})=\frac{1}{2}\}$. From the gap estimate in Corollary \ref{cor:gap}, it is easy to show that $\Sigma_{reg}$  is relatively open in $\partial\Om_{u}$. Moreover, if $x_{0}\in\Sigma_{reg}$ from the convergence properties of the blow-up, one can see that the domain is flat in a neighborhood of $x_{0}$, see also \cite[Proposition 6.2]{DET17Fre}. The function $u$ satisfies
$$\left\{
\begin{array}{ll}
\Delta_{g} u + \lambda_1(\Omega^*)\, u = 0&\textrm{on}\;\;\Omega^*\\[2mm]
u=0&\textrm{on}\;\;\partial\Omega^*\\[2mm]
\|\nabla^{g} u\|_{g}=\sqrt{\Lambda}&\textrm{on}\;\;\partial\Omega^*
\end{array}
\right.
$$
where the last equation is understood in the viscosity sense. We are then in position to apply \cite[Appendix A]{STV18Fre} (which is an adaptation of \cite{DS11Fre})
, which implies that near $x_{0}$ the set $\partial\{u>0\}=\partial\Om^*$ is $C^{1,\alpha}$. Indeed the equation $-\Delta_{g}u=\lambda_{1}(\Om^*)u$ can be written in divergence form 
\[ 
-\partial_{i}\left(g^{ij}\sqrt{|g|} \partial_j\right) u=\sqrt{|g|}\lambda_{1}(\Om^*)u\,.
\]
Moreover
\[
\|\nabla^{g} u \|_{g} =\sqrt{g(\nabla^{g} u, \nabla^{g} u)} = \sqrt{ (g^{ij} \partial_j u)^T\, g_{ij}\, (g^{ij} \partial_j u)} = \sqrt{(\nabla^e u)^T\, g^{ij}\, \nabla^e u}=
\]
\[
=  \sqrt{(\nabla^e u)^T\, (g^{ij})^{\frac12} \, (g^{ij})^{\frac12} \, \nabla^e u} = \sqrt{((g^{ij})^{\frac12} \, \nabla^e u)^T \, (g^{ij})^{\frac12} \, \nabla^e u} = |(g^{ij})^{\frac12}\, \nabla^e u|\,.
\]
Then, if we define the matrix $A = A_{ij} =g^{ij} \sqrt{|g|}$, we have that $u$ satisfies
$$\left\{
\begin{array}{ll}
-\div (A \cdot \nabla^e u) = \sqrt{|g|}\, \lambda_1(\Omega^*)\, u &\textrm{on}\;\;\Omega^*\\[2mm]
u=0&\textrm{on}\;\;\partial\Omega^*\\[2mm]
|A^{\frac12} \nabla^e u |=\sqrt{\Lambda} \, |g|^{\frac14}&\textrm{on}\;\;\partial\Omega^*
\end{array}
\right.
$$
and this allows us to use \cite[Appendix A]{STV18Fre}.
In order to obtain higher regularity for $\Sigma_{reg}$, we apply the classical results in \cite{KN77Reg}.
\item[(b)] We use a classical reduction of dimension argument, that can be found in two forms in the literature, namely the Federer's reduction principle (see for example \cite[Appendix A]{S83Lec}), or the approach of Weiss following \cite{G84Min}. We follow the latter, though we only give the sketch of the proof as it is very similar to \cite[Proposition 5.18]{MTV17Reg}, :
\begin{itemize}
\item if $n<k^*$, then any blow-up at $x_{0}\in\partial\Om_{u}$ is an 1-homogeneous global minimizer of the Alt-Caffarelli functional, and is therefore a half-plane, by definition of $k^*$, so it's density at the origin is $\frac12$, which means by \eqref{eq:density} that $\theta(x_{0})=\frac{1}{2}$, so $x_{0}\in\Sigma_{reg}$.
\item if $n=k^*$, let assume by contradiction that there is an infinite set of points $x_{n} \in \Sigma_{sing}:=\partial\Om_{u}\setminus\Sigma_{reg}$ in $M$. Up to subsequence, we can assume that $x_{n}$ converges to $x_{0}$, and as $\Sigma_{sing}$ is closed, we still have $x_{0}\in \Sigma_{sing}$. We denote $\eps_{n}=d_g(x_{n},x_{0})$ and consider the blow-up around $x_{0}$ done with the functions $u_{\eps_{n}}$, converging to $\Om_{0}$, a cone with singularity at $0$.\\
First we note that $\Om_{0}$ has only one singularity. Indeed, if it had another singularity, then by homogeneity it would have a line of singularity, which contradicts the results of Weiss (\cite[Theorem 4.1]{W99Par}). Therefore denoting $\xi_{n}:=\textnormal{exp}^g_{x_0} \Theta\left( \frac{x_{n}}{\eps_n}\right)\in\partial\Om_{u_{\eps_{n}}}$, converging (up to a subsequence) to $\xi_{0}\in\partial\Om_{0}$, we know that $\xi_{0}$ is a regular point of $\Om_{0}$. As a consequence, for $r_{0}$ small enough, $\phi^g_{u_{0},\xi_{0}}(r_{0})$ is close to $\frac{1}{2}$. By convergence of $u_{\eps_{n}}$ to $u_{0}$, this means that $\phi^{\bar g}_{u_{\eps_{n}},\xi_{0}}(r_{0})$ is close to $\frac{1}{2}$ for $n$ large enough. With similar computations as in \cite{MTV17Reg}, this implies that $\phi^{\bar g}_{u_{\eps_{n}},\xi_{n}}(r_{0})$ is close to $\frac{1}{2}$ for $n$ large enough. By monotonicity, this implies $\phi^{\bar g}_{u_{\eps_{n}},\xi_{n}}(r)$ is close to $\frac{1}{2}$ for small $r$, and in particular its limit when $r$ goes to 0, which is the density of $\xi_{n}$ in $\Om_{u_{\eps_{n}}}$. This is a contradiction as $\xi_{n}$ is a singular point.
\item Assume by contradiction that for some $s>n-k^*$ we have $\H^s(\Sigma_{sing})>0$. Then using again a blow-up analysis, and the density gap result, one can prove that there is $x_{0}\in\Sigma_{sing}$ and $\Om_{0}$ a blow-up at $x_{0}$ whose singular set also has a positive $\H^s$ measure (the details follow the same lines as in \cite[Proposition 5.18]{MTV17Reg}, itself relying on the strategy of \cite{W99Par}). This constitutes a contradiction and concludes the proof.\qed
\end{itemize}
\end{itemize}

\noindent\textit{Acknowledgements.} This work was partially supported by the project ANR-18-CE40-0013 SHAPO financed by the French Agence Nationale de la Recherche (ANR). 
P. Sicbaldi is partially supported by the grant \textquotedblleft Ram\'on y Cajal 2015\textquotedblright\, RYC-2015-18730 and the grant \textquotedblleft An\'alisis geom\'etrico\textquotedblright\, MTM 2017-89677-P. The authors would also like to thank B. Velichkov for valuable discussions on this work and for providing us an early version of \cite{RTV18Exi}.

\bibliographystyle{plain}
\bibliography{references}
\bigskip

\end{document}